%% file: main.tex
\documentclass[11pt,a4paper]{article}

\usepackage[margin=3cm]{geometry}
\usepackage{amsmath,amssymb,amsthm}
\usepackage{mathtools}
\usepackage{microtype}
\usepackage{xcolor}
\usepackage{tikz}
\usepackage[british]{babel}
\usepackage[hidelinks]{hyperref}
\usepackage[capitalise,noabbrev]{cleveref}

\theoremstyle{plain}
\newtheorem{theorem}{Theorem}
\newtheorem*{theorem*}{Theorem}

\newtheorem{lemma}[theorem]{Lemma}
\theoremstyle{definition}
\newtheorem{problem}{Problem}

\AddToHook{env/theorem/begin}{\crefalias{theorem}{theorem}}
\AddToHook{env/proposition/begin}{\crefalias{theorem}{proposition}}
\AddToHook{env/corollary/begin}{\crefalias{theorem}{corollary}}
\AddToHook{env/lemma/begin}{\crefalias{theorem}{lemma}}

\DeclareMathOperator{\CFI}{CFI}
\DeclareMathOperator{\tr}{tr}
\DeclareMathOperator{\val}{val}
\DeclareMathOperator{\supp}{supp}
\DeclareMathOperator{\rank}{rank}
\DeclareMathOperator{\diag}{diag}
\DeclareMathOperator{\spn}{span}
\DeclareMathOperator{\im}{im}
\DeclareMathOperator{\len}{len}

\newcommand{\R}{\mathbb{R}}
\newcommand{\Z}{\mathbb{Z}}
\newcommand{\calF}{\mathcal{F}}
\newcommand{\ones}{\mathbf{1}}

\title{Distinguishing graphs with simple spectrum by homomorphism counts}
\author{Takanori Maehara\thanks{\texttt{tmaehara@acm.org}}}
\date{21st October 2026}
\hypersetup{
  pdftitle={Distinguishing graphs with simple spectrum by homomorphism counts},
  pdfauthor={Takanori Maehara}
}

\begin{document}
\maketitle
\begin{abstract}
We study which graph classes distinguish any two non-isomorphic graphs
with simple spectrum by homomorphism counts from their members.
We show that such a class must have unbounded treewidth and unbounded
Euler genus, and must contain a graph with the complete graph on seven
vertices as a minor.
To prove this, we construct non-isomorphic cospectral pairs of bipartite,
nonsingular graphs with simple spectrum and maximum degree at most six
by attaching rooted trees to Cai--F\"urer--Immerman graphs.
The attachments preserve homomorphism indistinguishability over any
minor-closed class over which the original pair is indistinguishable.
We then prove that homomorphism counts from subcubic graphs identify
every graph of nullity at most one, extending the known result for
nonsingular graphs and covering all
graphs with simple spectrum.
The proof uses a linear-algebraic characterisation in which two graphs
have the same subcubic homomorphism counts if and only if they have the
same number of vertices and the images of their adjacency matrices are
related by a bijective linear map that
intertwines the adjacency operators and preserves both the inner products
and the trilinear products.
Finally, let $\calF$ consist of the single-vertex graph and all subdivisions
of the graphs in a sequence of connected cubic graphs with arbitrarily
large clique minors. We prove that homomorphism counts from $\calF$ have
the same distinguishing power as subcubic homomorphism counts.
\end{abstract}

{\small
\par}

\section{Introduction and Main Results}
\label{sec:intro}

Graphs are finite, simple, and undirected unless stated otherwise.
For a graph $G$, we write $A_G$ for its adjacency matrix and
$N_G(x)$ for the set of neighbours of $x$ in $G$.
We write
$\R^\times=\R\setminus\{0\}$ for the set of nonzero real numbers.
The \emph{nullity} of a graph is the dimension of the kernel of its
adjacency matrix.
A graph is \emph{nonsingular} if its nullity is zero.
A \emph{homomorphism} from a graph $F$ to a graph $G$
is a map from $V(F)$ to $V(G)$ that preserves adjacency. We write
$\hom(F,G)$ for the number of such maps. Two graphs $G$ and $G'$ are \emph{homomorphism
indistinguishable} over a graph class $\mathcal C$, written
$G\equiv_{\mathcal C}G'$, if
$\hom(F,G)=\hom(F,G')$ for every $F\in\mathcal C$.

Homomorphism counts characterise a range of graph properties and
equivalence relations. Lov\'asz~\cite{Lovasz1967} proved that two graphs
$G$ and $G'$ are isomorphic if and only if $\hom(F,G)=\hom(F,G')$ for every
graph $F$.
A graph is bipartite if and only if its homomorphism count from every odd
cycle is zero. Two graphs $G$ and $G'$ with the same number of vertices
are cospectral if and only if $\hom(C,G)=\hom(C,G')$ for every cycle $C$,
because the homomorphism count from the $k$-cycle to a graph is the trace
of the $k$-th power of its adjacency matrix. For every $k\geq1$, two graphs are
homomorphism indistinguishable over the graphs of treewidth at most $k$ if
and only if the $k$-dimensional Weisfeiler--Leman test does not distinguish
them~\cite{Dvorak2010,DellGroheRattan2018}.

A graph has \emph{simple spectrum} if its adjacency matrix has pairwise
distinct eigenvalues. We study homomorphism characterisations of
isomorphism on graphs with simple spectrum.
We call a class \emph{complete on simple spectrum} if any two graphs with
simple spectrum that are homomorphism indistinguishable over the class
are isomorphic.
By the result of Lov\'asz, the class of all graphs is complete on simple
spectrum. Motivated by the polynomial-time solvability of isomorphism
testing on graphs with simple spectrum~\cite{LeightonMiller1979,BabaiGrigoryevMount1982},
we ask whether smaller classes of test graphs can still be complete on
simple spectrum.

\begin{theorem}\label{thm:incompleteness}
Let $\mathcal C$ be any one of the following graph classes:
\begin{enumerate}
 \item graphs of treewidth at most $k$, for a fixed integer $k\geq1$;
 \item graphs embeddable on a surface of Euler genus at most $g$, for a fixed integer $g\geq0$;
 \item graphs with no $K_7$ minor.
\end{enumerate}
There are non-isomorphic graphs $G$ and $G'$ with simple spectrum such that
$G\equiv_{\mathcal C}G'$.
The graphs $G$ and $G'$ can be chosen cospectral, with each graph
bipartite, nonsingular, and of maximum degree at most six.
Consequently, every class complete on simple spectrum has unbounded treewidth
and unbounded Euler genus, and contains a graph with a $K_7$ minor.
\end{theorem}

The first case of \cref{thm:incompleteness} answers the question of Hordan, Dym, and
Seppelt~\cite[Section~7]{HordanDymSeppelt2026}, who proved this
incompleteness for multigraphs and asked whether it holds for simple graphs.
The second case covers planar graphs, which have unbounded treewidth;
so unbounded treewidth alone does not suffice for completeness.
The proof constructs a pair of graphs by attaching rooted trees to a pair
of Cai--F\"urer--Immerman (CFI) graphs.
The two resulting graphs have simple spectrum and are homomorphism
indistinguishable over every minor-closed class over which the original pair
is homomorphism indistinguishable
(\cref{lem:ss-transfer}).

The $K_7$ bound comes from the existing results used to obtain the initial
CFI pair~\cite{NeuenSeppelt2026,ThilikosWiederrecht2024}.
Whether every class complete on simple spectrum must contain
graphs with arbitrarily large clique minors remains open
(\cref{prob:clique-minors}). An affirmative answer would imply the weak
Roberson conjecture, stated in \cref{sec:open-problems}.

We next ask which graph classes suffice for completeness on simple
spectrum. A graph is \emph{subcubic} if its maximum degree is at most
three. Let $\mathcal D_3$ be the class of subcubic graphs.
There are non-isomorphic graphs that are homomorphism indistinguishable
over $\mathcal D_3$~\cite{Roberson2026}.
Nevertheless, the following theorem gives completeness on simple spectrum,
since a graph with simple spectrum has nullity at most one.

\begin{theorem}\label{thm:main}
 Let $G$ be a graph of nullity at most one. For every graph $G'$,
 $G\equiv_{\mathcal D_3}G'$ if and only if $G\cong G'$.
 In particular, $\mathcal D_3$ is complete on simple spectrum.
\end{theorem}

\Cref{thm:main} extends~\cite[Corollary~38]{CaiYoung2026} from
nonsingular graphs to graphs of nullity at most one.
To prove this theorem, we characterise subcubic homomorphism
indistinguishability for arbitrary graphs (\cref{lem:characterisation}):
$G\equiv_{\mathcal D_3}G'$ if and only if
$G$ and $G'$ have the same number of vertices and there is a bijective
linear map between the images of their adjacency matrices that intertwines the adjacency
operators and preserves both the inner products and the trilinear products.
For maximum degree three, \cref{lem:characterisation} provides a concrete
algebraic description of bounded-degree homomorphism indistinguishability,
as requested after Corollary~39 of~\cite{CaiYoung2026}.

It remains open whether the nullity bound in \cref{thm:main}
can be relaxed from one to two (\cref{prob:nullity-two}).
The bound cannot be raised to three: \cref{sec:open-problems} gives
non-isomorphic graphs of nullity three that are homomorphism
indistinguishable over $\mathcal D_3$.
On arbitrary graphs, deciding subcubic homomorphism indistinguishability
is at least as hard as graph isomorphism (\cref{sec:open-problems}).
We therefore ask whether a
subclass of $\mathcal D_3$ can be complete on simple spectrum while its
homomorphism indistinguishability relation is decidable in polynomial
time on arbitrary graphs (\cref{prob:polynomial-testing}).

Finally, we ask how far the subcubic class in \cref{thm:main} can be
restricted.
\begin{theorem}\label{thm:closure}
 Let $(H_N)_{N\geq1}$ be a sequence of connected cubic graphs such that,
 for every $t\geq1$, some $H_N$ contains $K_t$ as a minor.
 Let $\calF$ consist of $K_1$ and all subdivisions of the graphs $H_N$.
 Then
 \begin{equation}\label{eq:restriction-equivalence}
  G\equiv_{\calF}G'\Longleftrightarrow
  G\equiv_{\mathcal D_3}G'
 \end{equation}
 for all graphs $G$ and $G'$. In particular, $\calF$ is complete on simple spectrum.
\end{theorem}
Including $K_1$ in $\calF$ is necessary: $K_2$ and $K_2\sqcup K_1$ have simple
spectrum and equal homomorphism counts from every connected graph with at least two vertices.

The proof of \cref{thm:closure} embeds subdivisions of an arbitrary connected cubic multigraph
in the $H_N$ and interpolates in the lengths of the subdividing paths.
It compares contributions from connected components of $G$ and $G'$
in decreasing order of spectral radius.

\input{incompleteness}
\input{characterisation}
\input{restriction}
\input{open-problems}

\section*{Declarations}

The author used OpenAI GPT-6 Astra and Anthropic Claude Opus 5.5 in
preparing this work. The author formulated the problem and initial
theorem statements. GPT-6 Astra assisted with exploring the mathematics
and identifying relevant literature. The proofs were developed through
discussions between the author and GPT-6 Astra. GPT-6 Astra and Claude
Opus 5.5 both assisted with checking readability and improving the
presentation. The author assumes responsibility for all content.

\bibliographystyle{plain}
\bibliography{references}
\end{document}

%% file: incompleteness.tex
\section{Proof of Theorem~\ref{thm:incompleteness}}
\label{sec:incompleteness}

We use the variant of the CFI construction~\cite{CaiFurerImmerman1992} given
in~\cite[Definition~3.1]{Roberson2026}.
Let $B$ be a connected graph and fix a vertex $v_*$ of $B$, called the
\emph{twist vertex}.
For a vertex $v$ of $B$, the set of edges incident with $v$ is denoted by
$E_B(v)$.
For $i\in\{0,1\}$, let $\CFI_i(B)$ be the graph with
\begin{equation}\label{eq:ss-cfi-definition}
 \begin{aligned}
 \Xi_i(v)&=\left\{\xi\in\mathbb F_2^{E_B(v)}:
   \sum_{e\in E_B(v)}\xi_e=i\delta_{v,v_*}\right\},\\
 V(\CFI_i(B))&=\{(v,\xi):v\in V(B),\xi\in \Xi_i(v)\},\\
 E(\CFI_i(B))&=\{(v,\xi)(w,\zeta):vw\in E(B),\xi\in \Xi_i(v),\zeta\in \Xi_i(w),\xi_{vw}=\zeta_{vw}\}.
 \end{aligned}
\end{equation}
Here $\delta_{ab}$ is the Kronecker delta.
The vertices $(v,\xi)$ with $\xi\in \Xi_i(v)$ form a \emph{fibre} with
\emph{base vertex} $v$, and they are said to lie \emph{over} $v$.
The isomorphism type of $\CFI_1(B)$ is independent of the choice of
$v_*$, and the graphs $\CFI_0(B)$ and $\CFI_1(B)$ are
non-isomorphic~\cite[Lemma~5]{Neuen2024}.

For $i\in\{0,1\}$, let $\pi_i:\CFI_i(B)\to B$ be the projection
$\pi_i(v,\xi)=v$.
Let $F$ be a graph and $\chi:F\to B$ a homomorphism.
A homomorphism $f:F\to\CFI_i(B)$ is a \emph{lift} of $\chi$ if
$\pi_i\circ f=\chi$.
The number of such lifts is denoted by $\hom_\chi(F,\CFI_i(B))$.
By \cite[Theorem~3.6]{Roberson2026}, for every graph $F$ and every homomorphism
$\chi:F\to B$,
\begin{equation}\label{eq:ss-lift-equality}
 \hom(F,\CFI_0(B))=\hom(F,\CFI_1(B))
 \Longrightarrow
\hom_\chi(F,\CFI_0(B))=\hom_\chi(F,\CFI_1(B)).
\end{equation}

For each class $\mathcal C$ in \cref{thm:incompleteness}, we start
from a connected graph $B$ with
$\CFI_0(B)\equiv_{\mathcal C}\CFI_1(B)$.
For each base vertex, we choose a rooted graph and attach a copy at every
vertex lying over that base vertex in each CFI graph.
This preserves indistinguishability over $\mathcal C$
(\cref{sec:cfi-attachments}).
For a subdivided cubic base, we decompose the adjacency matrices of the
two rooted products into the same blocks, one for the whole base and one
for each edge of the original cubic graph (\cref{sec:cfi-blocks}).
We choose tree attachments that make these blocks nonsingular with simple,
pairwise disjoint spectra (\cref{sec:tree-selection}).
Replacing $B$ by such a base containing it as a minor then gives the
witnesses (\cref{sec:witness-construction}).

\subsection{Attaching rooted graphs to CFI graphs}
\label{sec:cfi-attachments}

Attaching a copy of a rooted graph at a vertex means identifying the
root of the copy with that vertex.
Each attached copy includes its root.
For each $z\in V(B)$, let $T_z$ be a rooted graph with root $r_z$.
The \emph{rooted product} $\widehat B$ is the graph obtained by
attaching one copy of $T_z$ at each $z\in V(B)$
and retaining all edges of $B$ \cite[Definition~1.1]{GodsilMcKay1978}.
For $i\in\{0,1\}$, the \emph{CFI-compatible rooted product}
$\widehat{\CFI}_i(B)$ is the graph obtained by attaching a separate copy
of $T_v$ at each $(v,\xi)\in V(\CFI_i(B))$ and retaining all edges of $\CFI_i(B)$.
The isomorphism type of $\widehat{\CFI}_i(B)$ is independent of the choice
of $v_*$.
Indeed, the fibre-preserving isomorphisms between CFI graphs with different
choices of $v_*$~\cite[proof of Lemma~3.2]{Roberson2026} extend to the rooted
products by mapping each attached copy identically onto the copy at the
image of its root.

The projection $\pi_i$ extends to
$\widehat\pi_i:\widehat{\CFI}_i(B)\to\widehat B$, mapping original CFI edges
to edges of $B$ and each attached copy of $T_z$ identically onto the copy
of $T_z$ in $\widehat B$.
See \cref{fig:ss-cfi-example}.

\input{figures/cfi-example}

\begin{lemma}\label{lem:ss-attachments}
Let $\mathcal C$ be a minor-closed graph class and $B$ a connected graph
such that $\CFI_0(B)\equiv_{\mathcal C}\CFI_1(B)$.
For any family of rooted graphs $(T_z)_{z\in V(B)}$, the CFI-compatible
rooted products satisfy
$\widehat{\CFI}_0(B)\equiv_{\mathcal C}\widehat{\CFI}_1(B)$.
\end{lemma}
\begin{proof}
Fix $F\in\mathcal C$.
We prove that, for $i\in\{0,1\}$,
\begin{equation}\label{eq:ss-attachment-count}
 \hom(F,\widehat{\CFI}_i(B))=
 \sum_{\eta:F\to\widehat B}\hom_{\chi_\eta}(F/F_\eta,\CFI_i(B)),
\end{equation}
where, for each homomorphism $\eta:F\to\widehat B$, the subgraph $F_\eta$ has
vertex set $V(F)$ and retains an edge of $F$ if and only if $\eta$ maps it
to an edge of one of the attached copies $T_z$ in $\widehat B$.
The homomorphism $\chi_\eta:F/F_\eta\to B$ is defined below.
Every vertex of $\widehat B$ lies in exactly one attached copy, so each
connected component of $F_\eta$ maps into one attached copy.
Consequently, every edge of $F$ whose endpoints lie in that component
belongs to $F_\eta$, so no edge outside $F_\eta$ becomes a loop upon contraction.
Here $F/F_\eta$ denotes the simple graph obtained from $F$ by contracting
each connected component of $F_\eta$ to one vertex and replacing parallel edges by single edges.
Thus $F/F_\eta$ is a minor of $F$.
Since $\mathcal C$ is minor-closed, $F/F_\eta\in\mathcal C$.
Since $\CFI_0(B)\equiv_{\mathcal C}\CFI_1(B)$ and $F/F_\eta\in\mathcal C$,
we have $\hom(F/F_\eta,\CFI_0(B))=\hom(F/F_\eta,\CFI_1(B))$.
Hence
$\hom_{\chi_\eta}(F/F_\eta,\CFI_0(B))
=\hom_{\chi_\eta}(F/F_\eta,\CFI_1(B))$
by \cref{eq:ss-lift-equality}.
Substituting this equality into \cref{eq:ss-attachment-count} proves the lemma.
It remains to define $\chi_\eta$ and to prove
\cref{eq:ss-attachment-count}.

The vertices of $F/F_\eta$ are the connected components of $F_\eta$.
Two distinct components are adjacent in $F/F_\eta$ if and only if an edge of
$F$ has one endpoint in each component.
For each connected component $K$ of $F_\eta$, define $\chi_\eta(K)$ to be
the unique $z\in V(B)$ whose copy of $T_z$ contains $\eta(K)$.
If $vw\in E(F)$ has $v\in K$ and $w\in L$ for distinct components $K,L$, then
$vw\notin E(F_\eta)$. Thus $\eta(v)\eta(w)$ is an edge of $B$ joining the
roots $\chi_\eta(K)$ and $\chi_\eta(L)$.
Hence $\chi_\eta:F/F_\eta\to B$ is a homomorphism.

To prove \cref{eq:ss-attachment-count}, we fix $\eta$ and construct a
bijection between homomorphisms
$f:F\to\widehat{\CFI}_i(B)$ satisfying $\widehat\pi_i\circ f=\eta$
and lifts $h:F/F_\eta\to\CFI_i(B)$ of $\chi_\eta$.
Fix such an $f$. Each connected component $K$ of $F_\eta$ maps
into a single attached copy, since every edge of $K$ projects into some
$T_z$. Let $h(K)$ be the root of this copy.
Then $\pi_i\circ h=\chi_\eta$, and edges between components map to edges
of $\CFI_i(B)$, so $h$ is a lift of $\chi_\eta$.
Conversely, fix a lift $h$ of $\chi_\eta$.
For each component $K$ of $F_\eta$, a copy of $T_{\chi_\eta(K)}$
is attached at $h(K)$ because $\pi_i(h(K))=\chi_\eta(K)$.
The projection $\widehat\pi_i$ maps this copy isomorphically onto
the copy of $T_{\chi_\eta(K)}$ in $\widehat B$.
For each $v\in K$, there is therefore a unique vertex $f(v)$ in the
copy rooted at $h(K)$ such that $\widehat\pi_i(f(v))=\eta(v)$.
Every edge of $F$ within $K$ lies in $F_\eta$ and is therefore preserved
by this isomorphism.
If $vw\in E(F)$ has $v\in K$ and $w\in L$ for distinct components $K,L$, then $\eta(v)$ and
$\eta(w)$ are roots. Hence $f(v)=h(K)$ and $f(w)=h(L)$, which are
adjacent because $h$ is a homomorphism.
The map $f\mapsto h$ and the map $h\mapsto f$ are inverse to each
other, which proves the bijection.
Summing over $\eta:F\to\widehat B$ proves \cref{eq:ss-attachment-count}.
\end{proof}

\subsection{Block-diagonalising the adjacency matrices}
\label{sec:cfi-blocks}

We seek attachments that make the CFI-compatible rooted products
nonsingular with simple spectrum.
Their adjacency matrices decompose into blocks given by weighted rooted
products.
We state this for the base used in the final construction, the subdivision
of a connected cubic graph
(\cref{lem:ss-cfi-blocks}).
Cubicity gives a simple block description, with one block over the
subdivided base and one for each edge of the cubic graph, while
subdivision permits bipartite witnesses of low degree.
It therefore suffices to make these rooted products nonsingular with
simple, pairwise disjoint spectra.

For a graph $\Gamma$, let $\Gamma^{\mathbf2}$ denote the subdivision
with one new vertex on each edge.
The subdivision $\Gamma^{\mathbf2}$ has vertex set $V(\Gamma^{\mathbf2})=V(\Gamma)\sqcup E(\Gamma)$,
where the new vertex on the edge $e$ is denoted by $e$.
Let $\Gamma^{\mathbf2}_0$ be $\Gamma^{\mathbf2}$ with every edge weighted $\sqrt2$, and for
$e=vw\in E(\Gamma)$ let $\Gamma^{\mathbf2}_e$ be the path with edges $ve,ew$,
both of weight $\sqrt2$.
For a family of rooted graphs $(T_z)_{z\in V(\Gamma^{\mathbf2})}$, write
$\widehat{\Gamma^{\mathbf2}_0}$ and $\widehat{\Gamma^{\mathbf2}_e}$ for the rooted products of
these weighted graphs with the corresponding attachments $T_z$.
The rooted products retain the original edge weights, and the edges of
each $T_z$ have weight one.

The following lemma identifies the blocks whose spectra we separate by
choosing rooted attachments.
The proof is based on constant and signed fibre vectors, which Hordan,
Dym, and Seppelt~\cite[Theorem~4]{HordanDymSeppelt2026} used to show that
$\CFI_0$ graphs over cubic bases have repeated eigenvalues.

\begin{lemma}\label{lem:ss-cfi-blocks}
Let $\Gamma$ be a connected cubic graph and let $(T_z)_{z\in V(\Gamma^{\mathbf2})}$
be a family of finite rooted graphs.
For each $i\in\{0,1\}$, the adjacency matrix of $\widehat{\CFI}_i(\Gamma^{\mathbf2})$
is orthogonally similar to
\begin{equation}\label{eq:ss-rooted-blocks}
 A_{\widehat{\Gamma^{\mathbf2}_0}}\oplus\bigoplus_{e\in E(\Gamma)}A_{\widehat{\Gamma^{\mathbf2}_e}}.
\end{equation}

\end{lemma}
\Cref{fig:ss-cfi-blocks} shows this decomposition for $\Gamma=K_4$.
\input{figures/cfi-blocks}
\begin{proof}
We prove the decomposition by specifying the corresponding subspaces
and orthonormal bases.
We choose the twist vertex $v_*$ among the original vertices $V(\Gamma)$,
so that every fibre over an inserted subdivision vertex has even parity.
Fix $i\in\{0,1\}$.
For $v\in V(\Gamma)$, the fibre over $v$ is indexed by $\Xi_i(v)$, which
consists of the four vectors in $\mathbb F_2^{E_{\Gamma^{\mathbf2}}(v)}$ of the
prescribed parity.
For each edge $e$ of $\Gamma$ incident with $v$, we write $\xi_e$ for $\xi_{ve}$.
For a subdivision vertex $e$, we have $e\ne v_*$, so its fibre has even parity
and is indexed by
$\Xi_i(e)=\{(0,0),(1,1)\}$; we write $b$ for $(b,b)$.
The vertices of $\widehat{\CFI}_i(\Gamma^{\mathbf2})$ are indexed by $(z,\xi,x)$, with $z\in V(\Gamma^{\mathbf2})$,
$\xi\in \Xi_i(z)$, and $x\in V(T_z)$.
The coordinate $(z,\xi,r_z)$ denotes the original CFI vertex $(z,\xi)$,
where $r_z$ is the root of $T_z$.

On $\Xi_i(v)$, we use the basis $\phi_{v,0}(\xi)=1/2$ and
$\phi_{v,e}(\xi)=(-1)^{\xi_e}/2$ for the three edges $e$ incident with $v$.
On $\Xi_i(e)$, we use the basis $\psi_{e,0}(b)=1/\sqrt2$ and
$\psi_{e,1}(b)=(-1)^b/\sqrt2$.
These bases are orthonormal under either parity constraint.
For any two distinct edges $e,e'$ incident with $v$, the pair
$(\xi_e,\xi_{e'})$ takes each value in $\mathbb F_2^2$ exactly once as
$\xi$ ranges over $\Xi_i(v)$, because parity determines the third coordinate.

For $z\in V(\Gamma^{\mathbf2})$, $g\in\R^{\Xi_i(z)}$, and $x\in V(T_z)$,
let $g[x]\in\R^{V(\widehat{\CFI}_i(\Gamma^{\mathbf2}))}$ be the vector whose
$(z,\xi,x)$-coordinate is $g(\xi)$ for each $\xi\in \Xi_i(z)$,
with all other coordinates zero.
The vectors $g[x]$ obtained from all basis vectors $g$ above and all
attachment positions $x$ form an orthonormal basis of $\R^{V(\widehat{\CFI}_i(\Gamma^{\mathbf2}))}$.
We define subspaces corresponding to $\widehat{\Gamma^{\mathbf2}_0}$
and the graphs $\widehat{\Gamma^{\mathbf2}_e}$.
Let $U_0$ be the subspace of vectors whose coordinates are constant in $\xi$
for each fixed $(z,x)$.
The subspace $U_0$ is spanned by $\phi_{v,0}[x]$ for $v\in V(\Gamma)$
and $x\in V(T_v)$, together with $\psi_{e,0}[x]$ for $e\in E(\Gamma)$
and $x\in V(T_e)$.
For each $e=vw$, let $U_e$ be spanned by $\phi_{v,e}[x]$, $\psi_{e,1}[x]$,
and $\phi_{w,e}[x]$, with $x$ ranging over $V(T_v)$, $V(T_e)$, and $V(T_w)$,
respectively.
These spanning sets partition the orthonormal basis.
For each subspace, we send each of its listed basis vectors $g[x]$ to the standard coordinate vector indexed
by $(z,x)$, where $g\in\R^{\Xi_i(z)}$.
This gives the isometric identifications
\begin{equation}\label{eq:ss-cfi-subspaces}
 \begin{aligned}
 U_0&\cong\R^{V(\widehat{\Gamma^{\mathbf2}_0})}, &
 U_e&\cong\R^{V(\widehat{\Gamma^{\mathbf2}_e})}\text{ for }e\in E(\Gamma).
 \end{aligned}
\end{equation}

We compute the adjacency matrix in this basis, starting with edges within the attached graphs.
For $z\in V(\Gamma^{\mathbf2})$, write $(g_\kappa)_{\kappa=1}^{|\Xi_i(z)|}$ for its fibre basis.
Let $Q_z$ be the matrix whose $\kappa$th column is $g_\kappa$, so
$(Q_z)_{\xi\kappa}=g_\kappa(\xi)$ for $\xi\in \Xi_i(z)$.
The subgraph of $\widehat{\CFI}_i(\Gamma^{\mathbf2})$ induced by the vertices $(z,\xi,x)$ with
$\xi\in \Xi_i(z)$ and $x\in V(T_z)$ is the disjoint union of $|\Xi_i(z)|$ copies
of $T_z$. Its adjacency matrix in coordinates $(\xi,x)$ is $I\otimes A_{T_z}$,
because the same graph is attached at every vertex of the fibre.
The simultaneous change of basis at all attachment positions is represented by
$Q_z\otimes I$, so
\begin{equation}\label{eq:ss-tree-blocks}
 (Q_z\otimes I)^{\top}(I\otimes A_{T_z})(Q_z\otimes I)
 =(Q_z^{\top}Q_z)\otimes A_{T_z}
 =I\otimes A_{T_z}.
\end{equation}
Thus the adjacency matrix of this induced subgraph has one block $A_{T_z}$
on each set $\{g_\kappa[x]:x\in V(T_z)\}$, with zero entries between these sets.
Each set lies in one of the subspaces $U_0,U_e$.

For an incidence $v\in e$, the root vertices
$(v,\xi,r_v)$ and $(e,b,r_e)$ are adjacent exactly when $\xi_e=b$.
Let $R_{v,e}$ be the matrix $R_{v,e}(\xi,b)=\delta_{\xi_e,b}$ of
this adjacency. It satisfies
\begin{equation}\label{eq:ss-fourier}
 \begin{aligned}
 R_{v,e}\psi_{e,0}&=\sqrt2\phi_{v,0}, &
 R_{v,e}\psi_{e,1}&=\sqrt2\phi_{v,e},
 \end{aligned}
\end{equation}
and, by orthonormality, $R_{v,e}^{\top}$ sends all other vectors
in the chosen basis of $\R^{\Xi_i(v)}$ to zero.
Thus the matrix of $R_{v,e}$ in these fibre bases has entry $\sqrt2$
at the row--column pairs $(\phi_{v,0},\psi_{e,0})$ and
$(\phi_{v,e},\psi_{e,1})$, and all other entries are zero.
Hence, in the stated bases of $U_0$ and $U_e$,
the submatrices indexed by root coordinates are $A_{\Gamma^{\mathbf2}_0}$ and $A_{\Gamma^{\mathbf2}_e}$,
respectively.

Combining these root-coordinate matrices with the attachment blocks in
\cref{eq:ss-tree-blocks} shows that the subspaces $U_0$ and $U_e$
are invariant and that the adjacency matrix of $\widehat{\CFI}_i(\Gamma^{\mathbf2})$ has the block-diagonal form
\eqref{eq:ss-rooted-blocks} in the stated bases.
\end{proof}

We therefore seek one family $(T_z)_{z\in V(\Gamma^{\mathbf2})}$ of rooted
graphs that makes $\widehat{\Gamma^{\mathbf2}_0}$ and all
$\widehat{\Gamma^{\mathbf2}_e}$ nonsingular with simple, pairwise disjoint spectra.

\subsection{Choosing tree attachments for simple and disjoint spectra}
\label{sec:tree-selection}

We choose attachments from a recursively constructed family of rooted trees
whose characteristic polynomials allow us to make the block spectra
simple and pairwise disjoint.

A \emph{rooted subcubic caterpillar} is a tree obtained from a nonempty path
by attaching at most one new leaf to each path vertex and choosing an
endpoint of the path as the root.
Thus the maximum degree is at most three and
the root degree is at most two.
The family $\mathcal T$ consists of the rooted subcubic caterpillars
obtained from $K_2$ rooted at an endpoint by zero or more applications
of the following two operations.
For a tree with root $r$, operation $X$ adds two new vertices $r',v$
and the edges $r'v,vr$, making $r'$ the new root.
Operation $Y$ adds a new root $r'$ and a new leaf $v$, with edges
$rr',r'v$ (\cref{fig:ss-operations}).

\begin{figure}[t]
\centering
\begin{tikzpicture}[
  line width=1pt,
  old/.style={circle,fill,inner sep=1.6pt},
  new/.style={circle,draw,fill=white,inner sep=1.6pt},
  tree/.style={draw,fill=black!8},
]
\begin{scope}
  \draw[tree] (0,0) -- (1.6,0.7) -- (1.6,-0.7) -- cycle;
  \node at (1.1,0) {$T$};
  \node[old,label=above:$r$] (r) at (0,0) {};
  \node[new] (u) at (-1,0) {};
  \node[new,label=above:$r'$] (s) at (-2,0) {};
  \draw (s) -- (u) -- (r);
  \draw (s) circle (4.5pt);
  \node at (-0.2,-1.3) {Operation $X$};
\end{scope}
\begin{scope}[xshift=5.5cm]
  \draw[tree] (0,0) -- (1.6,0.7) -- (1.6,-0.7) -- cycle;
  \node at (1.1,0) {$T$};
  \node[old,label=above:$r$] (r) at (0,0) {};
  \node[new,label=above:$r'$] (s) at (-1,0) {};
  \node[new] (l) at (-1,-0.8) {};
  \draw (l) -- (s) -- (r);
  \draw (s) circle (4.5pt);
  \node at (0.3,-1.3) {Operation $Y$};
\end{scope}
\end{tikzpicture}
\caption{The operations $X$ and $Y$ applied to a rooted subcubic
caterpillar $T$ with root $r$. New vertices are white, and the new
root $r'$ is circled.}
\label{fig:ss-operations}
\end{figure}

We compute the characteristic polynomial of each $T\in\mathcal T$
recursively along its construction.
For the root $r$ of $T$, let $p_T(\lambda)=\det(\lambda I-A_T)$ and
$q_T(\lambda)=\det(\lambda I-A_{T-r})$ be the characteristic polynomials of $T$ and
$T-r$, respectively.
For a nonzero $\xi=(p,q)^{\top}\in\R^2$, let
$[\xi]=[p:q]\in\mathbb P^1(\R)$ denote its span, identified with $p/q$ when $q\ne0$
and with $\infty$ when $q=0$.
For each $T\in\mathcal T$, put
$h_T(\lambda)=[p_T(\lambda):q_T(\lambda)]$.
We first show that these points are defined for every real $\lambda$, that is,
$p_T(\lambda)$ and $q_T(\lambda)$ never vanish together.
The starting pair for rooted $K_2$ is
$(p_{K_2}(\lambda),q_{K_2}(\lambda))=(\lambda^2-1,\lambda)$.
By cofactor expansion, the operations $X$ and $Y$
transform $(p_T(\lambda),q_T(\lambda))^{\top}$ by multiplication
by the respective matrices
\begin{equation}\label{eq:ss-recurrence}
 \begin{aligned}
 \mathbf{X}_\lambda&=\begin{pmatrix}\lambda^2-1&-\lambda\\\lambda&-1\end{pmatrix},
 &\mathbf{Y}_\lambda&=\begin{pmatrix}\lambda^2-1&-\lambda\\\lambda&0\end{pmatrix}.
 \end{aligned}
\end{equation}
At $\lambda=0$, the recurrence \eqref{eq:ss-recurrence} gives
$p_T(0)=(-1)^{|V(T)|/2}$ and $q_T(0)=0$.
For $\lambda\ne0$, the starting pair is nonzero and both
matrices in \cref{eq:ss-recurrence} are invertible, so they carry
nonzero pairs to nonzero pairs.
Thus $h_T(\lambda)$ is defined for every real $\lambda$.

We next use these projective values to determine the eigenvalue
multiplicities of the rooted-product blocks in \cref{eq:ss-rooted-blocks}.
The following lemma allows real edge weights in the base and gives a
kernel-dimension analogue of the determinant formula for unweighted rooted
products in \cite[Theorem~2.1]{GodsilMcKay1978}.

\begin{lemma}\label{lem:ss-rooted-multiplicity}
Let $H$ be a graph with real edge weights.
For each $z\in V(H)$, let $T_z\in\mathcal T$, and let $\widehat H$ be the
rooted product of $H$ with $(T_z)_{z\in V(H)}$.
For $\lambda\in\R$, define
\begin{equation}\label{eq:ss-rooted-core}
 K_H^\lambda=\diag(p_{T_z}(\lambda))_{z\in V(H)}
       -A_H\diag(q_{T_z}(\lambda))_{z\in V(H)}.
\end{equation}
Then, for every $\lambda\in\R$,
\begin{equation}\label{eq:ss-multiplicity}
 \dim\ker(\lambda I-A_{\widehat H})=\dim\ker K_H^\lambda.
\end{equation}
\end{lemma}
\begin{proof}
Fix $\lambda\in\R$ and put
$p_z=p_{T_z}(\lambda)$ and $q_z=q_{T_z}(\lambda)$.
Let $J=\{z\in V(H):q_z\ne0\}$ and $J^c=V(H)\setminus J$, and define
\begin{equation}\label{eq:ss-schur-core}
 \Sigma=\diag(p_z/q_z)_{z\in J}-A_{H}[J,J],
\end{equation}
where $A_{H}[J,J]$ is the principal submatrix indexed by $J$.
When $J=\varnothing$, we have $\dim\ker \Sigma=0$.
We will show that both sides of \cref{eq:ss-multiplicity} equal
$\dim\ker \Sigma$.

We first consider $K_H^\lambda$.
For $z\in J^c$, we have $q_z=0$ and hence $p_z\ne0$, because the two
polynomials defining $h_{T_z}$ have no common real zero.
Thus column $z$ has the single
nonzero entry $p_z$ in row $z$.
Ordering $J$ before $J^c$ therefore gives
\begin{equation}\label{eq:ss-core-partition}
 K_H^\lambda=
 \begin{pmatrix}
 K_H^\lambda[J,J]&0\\
 K_H^\lambda[J^c,J]&\diag(p_z)_{z\in J^c}
 \end{pmatrix}.
\end{equation}
Since the lower-right block is invertible,
$\dim\ker K_H^\lambda=\dim\ker K_H^\lambda[J,J]$.
Also $K_H^\lambda[J,J]=\Sigma\diag(q_z)_{z\in J}$ with every $q_z\ne0$,
so $\dim\ker K_H^\lambda=\dim\ker \Sigma$.

We next prove $\dim\ker(\lambda I-A_{\widehat H})=\dim\ker \Sigma$.
Let $\widehat{H[J]}$ be the rooted product of the induced base graph
$H[J]$ with $(T_z)_{z\in J}$.
We first prove the deletion identity
\begin{equation}\label{eq:ss-tree-deletion}
 \dim\ker(\lambda I-A_{\widehat H})
 =\dim\ker(\lambda I-A_{\widehat{H[J]}}).
\end{equation}
To prove \cref{eq:ss-tree-deletion}, we delete from $\widehat H$
the whole copies $T_z$ with $z\in J^c$, including their roots, one at a time.
At each step, let $L$ be the remaining induced subgraph of $\widehat H$.
For the copy $T_z$ to be removed, put $D_z=\lambda I-A_{T_z}$ and
$R_z=\lambda I-A_{L-V(T_z)}$.
Let $e_z$ be the root coordinate vector in $\R^{V(T_z)}$, and let
$b_z$ be the vector of edge weights from the root to $V(L)\setminus V(T_z)$.
Ordering $V(T_z)$ first gives
\begin{equation}\label{eq:ss-whole-tree-block}
 \lambda I-A_L=
 \begin{pmatrix}D_z&-e_zb_z^{\top}\\-b_ze_z^{\top}&R_z\end{pmatrix}.
\end{equation}
Here $D_z$ is invertible because $\det D_z=p_z\ne0$.
Cramer's rule gives $e_z^{\top}D_z^{-1}e_z=q_z/p_z=0$.
Thus the Schur complement of $D_z$ in \cref{eq:ss-whole-tree-block} is
\begin{equation}\label{eq:ss-zero-tree-correction}
 R_z-b_z(e_z^{\top}D_z^{-1}e_z)b_z^{\top}
 =R_z-\frac{q_z}{p_z}b_zb_z^{\top}=R_z.
\end{equation}
Since $D_z$ is invertible,
$\dim\ker(\lambda I-A_L)=\dim\ker R_z$.
The remaining matrix is exactly $R_z=\lambda I-A_{L-V(T_z)}$,
so the same calculation applies after each deletion.
Deleting all copies indexed by $J^c$ therefore proves
\cref{eq:ss-tree-deletion}.

It remains to show that
$\dim\ker(\lambda I-A_{\widehat{H[J]}})=\dim\ker \Sigma$.
For $z\in J$, put $E_z=\lambda I-A_{T_z-r_z}$, and let
$c_z$ be the vector of edge weights from $r_z$ to the non-root vertices
of $T_z$.
Let $E=\bigoplus_{z\in J}E_z$, and let $C$ have rows indexed by these
non-root vertices and columns indexed by $J$.
Column $z$ of $C$ equals $c_z$ on $V(T_z)\setminus\{r_z\}$ and is
zero elsewhere.
Ordering all non-root vertices before the roots gives
\begin{equation}\label{eq:ss-nonroot-block}
 \lambda I-A_{\widehat{H[J]}}=
 \begin{pmatrix}E&-C\\-C^{\top}&\lambda I-A_{H}[J,J]\end{pmatrix}.
\end{equation}
Every $E_z$ is invertible because $\det E_z=q_z\ne0$.
The determinant formula for the Schur complement of $E_z$ in
$\lambda I-A_{T_z}$ gives
$p_z=q_z(\lambda-c_z^{\top}E_z^{-1}c_z)$.
Since distinct columns of $C$ are supported on distinct blocks of $E$,
the Schur complement of $E$ in \cref{eq:ss-nonroot-block} is
\begin{equation}\label{eq:ss-root-schur}
 \lambda I-A_{H}[J,J]-C^{\top}E^{-1}C
 =\diag(p_z/q_z)_{z\in J}-A_{H}[J,J]=\Sigma.
\end{equation}
Since $E$ is invertible, \cref{eq:ss-tree-deletion,eq:ss-root-schur} give
\begin{equation}\label{eq:ss-kernel-reduction}
 \dim\ker(\lambda I-A_{\widehat H})
 =\dim\ker \Sigma
 =\dim\ker K_H^\lambda,
\end{equation}
which proves \cref{eq:ss-multiplicity}.
\end{proof}

By \cref{lem:ss-rooted-multiplicity}, simple and pairwise disjoint block spectra
require the sum of the nullities of the corresponding matrices $K_H^\lambda$
to be at most one for every $\lambda\in\R$.
These nullities depend only on the projective attachment values
$h_{T_z}(\lambda)$.
Choosing a tree $T$ fixes the whole function $\lambda\mapsto h_T(\lambda)$,
so one tree must avoid the forbidden values at every relevant $\lambda$ at once.

\begin{lemma}\label{lem:ss-caterpillars}
For any finite list of pairs
$(\lambda_j,a_j)\in\R^\times\times\mathbb P^1(\R)$, there is a tree
$T\in\mathcal T$ with $h_T(\lambda_j)\ne a_j$
for every $j$.
\end{lemma}
\begin{proof}
Let $\omega$ be a finite composition of the construction operations
$X$ and $Y$, and let $T=\omega(K_2)$.
For each $j$, let $M_{\omega,j}$ be the matrix product obtained by
replacing each $X$ by $\mathbf{X}_{\lambda_j}$, each $Y$ by
$\mathbf{Y}_{\lambda_j}$, and composition by matrix multiplication.
For the empty composition, $M_{\omega,j}=I$.
Let $\nu_j=(p_{K_2}(\lambda_j),q_{K_2}(\lambda_j))^{\top}
=(\lambda_j^2-1,\lambda_j)^{\top}$.
Then $(p_T(\lambda_j),q_T(\lambda_j))^{\top}=M_{\omega,j}\nu_j$
by \cref{eq:ss-recurrence}.
Therefore, it suffices to find a composition $\omega$ such that
$[M_{\omega,j}\nu_j]\ne a_j$ for every $j$.

We express the conditions $[M_{\omega,j}\nu_j]\ne a_j$ for all $j$
as the nonvanishing of a linear functional on a tensor product.
For each $j$, define the linear functional
$\ell_j:\R^2\to\R$ by
\begin{equation}\label{eq:ss-avoidance-factors}
 \ell_j(p,q)=
 \begin{cases}
  p-a_jq,&a_j\in\R,\\
  q,&a_j=\infty.
 \end{cases}
\end{equation}
Then $\ell_j(M_{\omega,j}\nu_j)\ne0$ exactly when
$h_T(\lambda_j)\ne a_j$.
Let
$\nu=\bigotimes_j\nu_j$, $\ell=\bigotimes_j\ell_j$,
$\mathbf{X}=\bigotimes_j\mathbf{X}_{\lambda_j}$, and
$\mathbf{Y}=\bigotimes_j\mathbf{Y}_{\lambda_j}$,
where $\otimes$ is the tensor product.
Define $M_\omega$ in the same way using $\mathbf{X},\mathbf{Y}$.
Multiplication of tensor products gives
$M_\omega=\bigotimes_jM_{\omega,j}$.
Since $\ell$ is a tensor product of functionals,
\begin{equation}\label{eq:ss-positive-word}
 \ell(M_\omega\nu)=\prod_j\ell_j(M_{\omega,j}\nu_j).
\end{equation}
Thus the avoidance conditions hold exactly when $\ell(M_\omega\nu)\ne0$.

It suffices to find a finite product $W$ of the matrices
$\mathbf{X}^{\pm1},\mathbf{Y}^{\pm1}$ with $\ell(W\nu)\ne0$.
Since every $\lambda_j\ne0$, the matrices $\mathbf{X},\mathbf{Y}$ are
invertible. By the Cayley--Hamilton theorem, their inverses are
polynomials in $\mathbf{X}$ and $\mathbf{Y}$, respectively.
Substituting these polynomial expressions into $W$ and expanding gives
a finite sum $W=\sum_i c_iM_{\omega_i}$, where $c_i\in\R$ and each
$\omega_i$ is a composition of the construction operations $X$ and $Y$.
Consequently, $\ell(W\nu)=\sum_i c_i\ell(M_{\omega_i}\nu)\ne0$
implies $\ell(M_{\omega_i}\nu)\ne0$ for some $i$.
By \cref{eq:ss-positive-word}, the tree $T=\omega_i(K_2)$ then satisfies
$h_T(\lambda_j)\ne a_j$ for every $j$.

We now construct such a product $W$.
For $\lambda\ne0$, the matrices in \cref{eq:ss-recurrence} act on
$\mathbb P^1(\R)$ by multiplication of vector representatives.
The product $\mathbf{X}_\lambda^{-1}\mathbf{Y}_\lambda$ acts by
$t\mapsto\lambda^{-2}t+\lambda^{-1}$.
The starting value is $h_{K_2}(\lambda)=\lambda-\lambda^{-1}=\mathbf{Y}_\lambda\infty$.
Since $\mathbf{Y}_\lambda^{-1}\infty=0$, we have
$\mathbf{Y}_\lambda^{-2}h_{K_2}(\lambda)=0$.
Thus $(\mathbf{X}_\lambda^{-1}\mathbf{Y}_\lambda)^m\mathbf{Y}_\lambda^{-2}$ sends the starting value to
$\lambda^{-1}\sum_{k=0}^{m-1}\lambda^{-2k}$.
These values are strictly monotone in $m$ for every real $\lambda\ne0$,
so each $a_j$ is attained for at most one $m$ at $\lambda=\lambda_j$.
Choose a positive integer $m$ such that
$[(\mathbf{X}_{\lambda_j}^{-1}\mathbf{Y}_{\lambda_j})^m\mathbf{Y}_{\lambda_j}^{-2}\nu_j]\ne a_j$ for every $j$.
Then $W=(\mathbf{X}^{-1}\mathbf{Y})^m\mathbf{Y}^{-2}$ satisfies
$\ell(W\nu)=\prod_j\ell_j((\mathbf{X}_{\lambda_j}^{-1}\mathbf{Y}_{\lambda_j})^m\mathbf{Y}_{\lambda_j}^{-2}\nu_j)\ne0$,
as required.
\end{proof}

We now combine the multiplicity formula and the avoidance property to choose
the tree attachments for all blocks simultaneously.

\begin{lemma}\label{lem:ss-separation}
Let $Z$ be a finite set, and let $H_1,\ldots,H_s$ be connected graphs
with nonzero real edge weights and pairwise distinct nonempty vertex sets
$Z_i=V(H_i)\subseteq Z$.
There is a family $(T_z)_{z\in Z}$ of trees in $\mathcal T$ such that the rooted products
of $H_i$ with $(T_z)_{z\in Z_i}$ are nonsingular and have simple, pairwise
disjoint spectra.
\end{lemma}
\begin{proof}
Let $A_{H_i}$ be the weighted adjacency matrix of $H_i$.
For $\alpha,\beta\in\mathbb C^Z$, let
\begin{equation}\label{eq:ss-core}
 K_{H_i}(\alpha,\beta)=\diag(\alpha_z)_{z\in Z_i}-A_{H_i}\diag(\beta_z)_{z\in Z_i}.
\end{equation}
For a projective point $a=([\alpha_z:\beta_z])_{z\in Z}$,
rescaling each homogeneous representative multiplies $K_{H_i}(\alpha,\beta)$
on the right by an invertible diagonal matrix.
Thus $\dim\ker K_{H_i}(\alpha,\beta)$ depends only on $a$.

We define the following closed subsets of $(\mathbb P^1(\mathbb C))^Z$,
for $k\in\{1,2\}$:
\begin{equation}\label{eq:ss-bad-sets}
 \begin{aligned}
 \mathcal V_i^{(k)}&=\{([\alpha_z:\beta_z])_{z\in Z}:\dim\ker K_{H_i}(\alpha,\beta)\ge k\}, &
 \mathcal E&=\bigcup_i\mathcal V_i^{(2)}
       \cup\bigcup_{i\ne i'}(\mathcal V_i^{(1)}\cap \mathcal V_{i'}^{(1)}).
 \end{aligned}
\end{equation}
Fix trees $T_z\in\mathcal T$, and write $\widehat H_i$ for the rooted
product of $H_i$ with $(T_z)_{z\in Z_i}$.
For $\lambda\in\R$, let
$K_{H_i}^\lambda=K_{H_i}((p_{T_z}(\lambda))_{z\in Z},(q_{T_z}(\lambda))_{z\in Z})$.
By \cref{lem:ss-rooted-multiplicity}, for $\lambda\in\R$ and
$a=(h_{T_z}(\lambda))_{z\in Z}$, the number $\lambda$ is an eigenvalue
of $\widehat H_i$ of multiplicity at least $k$ exactly when
$a\in\mathcal V_i^{(k)}$.
It is an eigenvalue of both $\widehat H_i$ and $\widehat H_{i'}$ exactly
when $a\in\mathcal V_i^{(1)}\cap\mathcal V_{i'}^{(1)}$.
At $\lambda=0$, the matrix $K_{H_i}^0$ is invertible
because $q_{T_z}(0)=0$ and $p_{T_z}(0)\ne0$.
Thus the products $\widehat H_i$ are nonsingular.
Their spectra are simple and pairwise disjoint if
\begin{equation}\label{eq:ss-bad}
 (h_{T_z}(\lambda))_{z\in Z}\notin\mathcal E
 \quad\text{for every }\lambda\in\R^\times.
\end{equation}

Put $d=|Z|$.
We next prove $\dim\mathcal E\le d-2$ over $\mathbb C$.
The polynomial $\det(\diag(t)-A_{H_i})$ in the variables $(t_z)_{z\in Z_i}$
is irreducible because $H_i$ is connected
\cite[Lemma~4.1]{AlAhmadieh2026}.
Substituting $t_z=\alpha_z/\beta_z$ and clearing denominators gives the
determinant of the matrix in \cref{eq:ss-core}, as a polynomial in the
homogeneous representatives $(\alpha_z,\beta_z)$.
The monomial $\prod_{z\in Z_i}\alpha_z$ shows that none of the $\beta_z$
divides this polynomial. Together with the irreducibility of its
dehomogenisation, this implies that the multihomogeneous polynomial is irreducible.
Its projective zero set is $\mathcal V_i^{(1)}$.
The coordinates indexed by $Z\setminus Z_i$ are unrestricted, so
$\mathcal V_i^{(1)}$ is an irreducible hypersurface in $(\mathbb P^1(\mathbb C))^Z$.
For each $z\in Z$, let $a^{(z)}\in(\mathbb P^1(\mathbb C))^Z$ have
coordinate $[0:1]$ at $z$ and $[1:0]$ elsewhere.
At $a^{(z)}$, if $z\in Z_i$, all columns of $K_{H_i}(\alpha,\beta)$ other than column $z$
are standard basis vectors, and row $z$ is zero, so the matrix has kernel dimension one.
If $z\notin Z_i$, the matrix is the identity, so its kernel has dimension zero.
For $i\ne i'$, choose $z$ in the symmetric difference of $Z_i$ and
$Z_{i'}$; then $a^{(z)}$ lies in exactly one of $\mathcal V_i^{(1)}$ and
$\mathcal V_{i'}^{(1)}$.
Thus the $\mathcal V_i^{(1)}$ are distinct, so each intersection
$\mathcal V_i^{(1)}\cap \mathcal V_{i'}^{(1)}$ has dimension at most $d-2$.
For each $i$ and each $z\in Z_i$, the point $a^{(z)}$ lies in
$\mathcal V_i^{(1)}\setminus\mathcal V_i^{(2)}$.
Thus $\mathcal V_i^{(2)}$ is a proper closed subset of
$\mathcal V_i^{(1)}$ and also has dimension at most $d-2$.
The union defining $\mathcal E$ therefore has dimension at most $d-2$.

To satisfy \cref{eq:ss-bad}, we choose trees successively.
We enumerate $Z$ as $\{1,\ldots,d\}$.
For fixed trees $T_1,\ldots,T_j$, let $P_j$ be the set of pairs
$(\lambda,a)$ with $\lambda\in\R^\times$,
$a\in\mathcal E$, and
$a_z=h_{T_z}(\lambda)$ for all $z\le j$.
Let $\Omega_j$ be the Zariski closure of $P_j$ in
$\mathbb C\times\mathcal E$.
Then $\dim\Omega_0\le d-1$.
We choose the trees so that the dimension of $\Omega_j$ strictly decreases
whenever it is nonempty; after $d$ choices, the set will be empty.
We choose $T_{j+1}$ so that $\Omega_{j+1}$ contains no irreducible
component of $\Omega_j$.
By definition, $P_{j+1}\subseteq P_j$, so $\Omega_{j+1}\subseteq\Omega_j$.
Every pair $(\lambda,a)\in P_{j+1}$ satisfies
$a_{j+1}=h_{T_{j+1}}(\lambda)$. In homogeneous coordinates, this is
equivalent to
\begin{equation}\label{eq:ss-cut}
 \alpha_{j+1}q_{T_{j+1}}(\lambda)-\beta_{j+1}p_{T_{j+1}}(\lambda)=0.
\end{equation}
This polynomial equation also holds on the Zariski closure
$\Omega_{j+1}$.
It therefore suffices to choose $T_{j+1}$ so that \cref{eq:ss-cut}
fails at one point of every irreducible component of $\Omega_j$.

For each irreducible component $\mathcal W$ of $\Omega_j$, choose a pair
$(\lambda_{\mathcal W},a_{\mathcal W})\in P_j\cap\mathcal W$.
Such a pair exists because $P_j$ is dense in $\Omega_j$, which has only
finitely many irreducible components.
The pairs $(\lambda_{\mathcal W},(a_{\mathcal W})_{j+1})$ form a finite
list in $\R^\times\times\mathbb P^1(\mathbb C)$.
A nonreal projective point cannot equal $h_T(\lambda)$ for real $\lambda$.
Applying \cref{lem:ss-caterpillars} to those pairs in the list whose
projective coordinate is real therefore gives a tree $T_{j+1}\in\mathcal T$ with
$h_{T_{j+1}}(\lambda_{\mathcal W})\ne(a_{\mathcal W})_{j+1}$ for every $\mathcal W$.
Thus \cref{eq:ss-cut} fails at $(\lambda_{\mathcal W},a_{\mathcal W})$ in every irreducible
component $\mathcal W$ of $\Omega_j$.
Thus $\Omega_{j+1}$ contains no irreducible component of $\Omega_j$,
so $\dim\Omega_{j+1}<\dim\Omega_j$.
If $\Omega_j$ is empty, we choose the remaining trees arbitrarily.
Since $\dim\Omega_0\le d-1$ and the dimension strictly decreases
whenever $\Omega_j$ is nonempty, $\Omega_d=\varnothing$, so $P_d=\varnothing$.
Hence \cref{eq:ss-bad} holds, and the lemma follows.
\end{proof}

\subsection{Constructing the witnesses}
\label{sec:witness-construction}

We now combine the preceding lemmas to prove \cref{lem:ss-transfer},
which converts CFI counterexamples into counterexamples with simple spectrum.
\Cref{thm:incompleteness} then follows by applying this lemma to known CFI
pairs for its three graph classes.

\begin{lemma}\label{lem:ss-cfi-spectrum}
Let $\Gamma$ be a connected cubic graph.
There is a family $(T_z)_{z\in V(\Gamma^{\mathbf2})}$ of trees in $\mathcal T$ such that the
CFI-compatible rooted products $\widehat{\CFI}_0(\Gamma^{\mathbf2})$ and $\widehat{\CFI}_1(\Gamma^{\mathbf2})$
are non-isomorphic, bipartite, cospectral, nonsingular graphs with
simple spectrum and maximum degree at most six.
\end{lemma}
\begin{proof}
Recall that $\Gamma^{\mathbf2}_0$ is $\Gamma^{\mathbf2}$ with every edge
weighted $\sqrt2$, and $\Gamma^{\mathbf2}_e$ is the path $v$--$e$--$w$
with both edge weights $\sqrt2$, for $e=vw\in E(\Gamma)$.
They are connected and have nonzero edge weights. Their vertex sets
$V(\Gamma^{\mathbf2})$ and $\{v,e,w\}$ for $e=vw\in E(\Gamma)$ are pairwise distinct.
By \cref{lem:ss-separation} with $Z=V(\Gamma^{\mathbf2})$,
there is one tree $T_z\in\mathcal T$ for each
$z\in V(\Gamma^{\mathbf2})$ such that the rooted products $\widehat{\Gamma^{\mathbf2}_0}$ and
$\widehat{\Gamma^{\mathbf2}_e}$ are nonsingular and have simple, pairwise disjoint
spectra.
The block matrix in \cref{eq:ss-rooted-blocks} therefore has simple
spectrum and is nonsingular.
By \cref{lem:ss-cfi-blocks}, the adjacency matrices of
$\widehat{\CFI}_0(\Gamma^{\mathbf2})$ and $\widehat{\CFI}_1(\Gamma^{\mathbf2})$ are both
orthogonally similar to this matrix, so the two rooted products are
cospectral and nonsingular and have simple spectrum.

Each CFI vertex over $v\in V(\Gamma)$ has one neighbour over each of the
three adjacent subdivision vertices, hence degree three.
Each CFI vertex over $e=vw\in E(\Gamma)$ has two neighbours over each
of $v,w$, hence degree four.
Repeatedly deleting leaves from either rooted product removes exactly
the non-root vertices of its attached trees.
The graphs $\widehat{\CFI}_0(\Gamma^{\mathbf2})$ and $\widehat{\CFI}_1(\Gamma^{\mathbf2})$ are
non-isomorphic because their maximal subgraphs of minimum degree at least two
are the non-isomorphic graphs $\CFI_0(\Gamma^{\mathbf2})$ and $\CFI_1(\Gamma^{\mathbf2})$.
Both rooted products are bipartite, since each CFI graph is bipartite
and attaching trees preserves bipartiteness.
Each attached tree is subcubic with root degree at most two, so the rooted products have
maximum degree at most six.
\end{proof}

\begin{lemma}\label{lem:ss-transfer}
Let $\mathcal C$ be a minor-closed graph class and let $B$ be a connected
graph satisfying $\CFI_0(B)\equiv_{\mathcal C}\CFI_1(B)$. Then there are non-isomorphic
graphs $G$ and $G'$ with simple spectrum such that $G\equiv_{\mathcal C}G'$.
The witnesses can be chosen bipartite, cospectral, nonsingular, and of
maximum degree at most six.
\end{lemma}

\begin{proof}
To apply \cref{lem:ss-cfi-spectrum}, we first construct a connected cubic
graph $\Gamma$ such that $\Gamma^{\mathbf2}$ contains $B$ as a minor.
Let $t$ be an integer with $t\geq\max\{4,|V(B)|\}$.
Let $\Gamma$ be the connected cubic graph obtained from $K_t$ by
replacing each vertex with a cycle of length $t-1\geq3$ and attaching its
$t-1$ incident edges to distinct cycle vertices.
Subdividing edges preserves connectedness, so $\Gamma^{\mathbf2}$ is connected.
Contracting one edge at each subdivision vertex of $\Gamma^{\mathbf2}$ gives $\Gamma$,
and contracting the replacement cycles of $\Gamma$ gives $K_t$.
Since $t\geq|V(B)|$, the graph $B$ is a minor of $K_t$ and hence of $\Gamma^{\mathbf2}$.
By \cite[Corollary~5.7]{Roberson2026}, CFI indistinguishability over the
minor-closed class $\mathcal C$ passes from a connected base to a connected
base containing it as a minor. Hence the assumption
$\CFI_0(B)\equiv_{\mathcal C}\CFI_1(B)$ gives
$\CFI_0(\Gamma^{\mathbf2})\equiv_{\mathcal C}\CFI_1(\Gamma^{\mathbf2})$.
Using the trees from \cref{lem:ss-cfi-spectrum}, let
$G=\widehat{\CFI}_0(\Gamma^{\mathbf2})$ and $G'=\widehat{\CFI}_1(\Gamma^{\mathbf2})$.
By \cref{lem:ss-cfi-spectrum}, it remains to prove
$G\equiv_{\mathcal C}G'$, which follows from \cref{lem:ss-attachments}.
\end{proof}

\begin{proof}[Proof of \cref{thm:incompleteness}]
Each of the three classes is minor-closed.
By \cref{lem:ss-transfer}, it suffices to give, for each of the three
classes $\mathcal C$, a connected graph $B$ with
$\CFI_0(B)\equiv_{\mathcal C}\CFI_1(B)$.
When $\mathcal C$ is the class of graphs of treewidth at most $k\geq1$, the graph
$B=K_{k+2}$ satisfies
$\CFI_0(B)\equiv_{\mathcal C}\CFI_1(B)$~\cite[Theorem~10 and Lemma~12]{Neuen2024}.
When $\mathcal C$ is the class of graphs of Euler genus at most $g$, a connected
graph $B$ with $\CFI_0(B)\equiv_{\mathcal C}\CFI_1(B)$ is given in
\cite[Lemma~2.10 and proof of Theorem~1.8]{NeuenSeppelt2026}.
When $\mathcal C$ is the class of $K_7$-minor-free graphs, a connected graph $B$ with
$\CFI_0(B)\equiv_{\mathcal C}\CFI_1(B)$ is given in
\cite[Lemma~2.10 and proof of Theorem~1.9]{NeuenSeppelt2026}
and \cite[Theorem~1.7 and Section~1.3]{ThilikosWiederrecht2024}.
\end{proof}

%% file: figures/cfi-example.tex
\begin{figure}[t]
\centering
\begin{tikzpicture}[
  scale=.95,
  base/.style={circle,fill,inner sep=1.8pt},
  cfi/.style={circle,fill,inner sep=1.5pt},
  att/.style={circle,draw,fill=white,inner sep=1.3pt},
  lab/.style={font=\scriptsize,inner sep=1.5pt},
  line width=.8pt,
  every node/.style={font=\small},
]
\def\treeA#1#2{\draw (#1)--++(#2:.26) node[att]{};}
\def\treeB#1#2{\draw (#1)--++(#2:.26) node[att]{}--++(#2:.26) node[att]{};}
\def\treeC#1#2{\draw (#1)--++(#2:.26) node[att]{}--++(#2:.26) node[att]{}
                --++(#2:.26) node[att]{};}
\pgfmathsetmacro{\baseAngle}{atan2(.95/3,.6)}
\pgfmathsetmacro{\cfiAngle}{atan2(.6,.45)}
\foreach \attached in {0,1} {
  \begin{scope}[xshift={5.8*\attached cm}]
    \coordinate (a) at (0,.55);
    \coordinate (b) at (-.6,-.4);
    \coordinate (c) at (.6,-.4);
    \draw (a)--(b)--(c)--cycle;
    \ifnum\attached=0
      \node[base,label={[lab]above:$a$}] at (a) {};
      \node[base,label={[lab]below left:$b$}] at (b) {};
      \node[base,label={[lab]below right:$c$}] at (c) {};
      \node at (0,-1.55) {$B$};
    \else
      \treeA{a}{90}
      \treeB{b}{180+\baseAngle}
      \treeC{c}{360-\baseAngle}
      \node[lab,above] at (0,.85) {$T_a$};
      \node[lab,above] at (-.95,-.48) {$T_b$};
      \node[lab,above] at (1.05,-.5) {$T_c$};
      \foreach \v in {a,b,c} \node[base] at (\v) {};
      \node at (0,-1.55) {$\widehat B$};
    \fi
    \begin{scope}[xshift={(2.5+.8*\attached)*1cm}]
      \coordinate (a0) at (-1.4,.075);
      \coordinate (b0) at (-.45,.675);
      \coordinate (c0) at (.45,.675);
      \coordinate (a1) at (1.4,.075);
      \coordinate (b1) at (.45,-.525);
      \coordinate (c1) at (-.45,-.525);
      \draw (a0)--(b0)--(c0)--(a1)--(b1)--(c1)--cycle;
      \ifnum\attached=0
        \node[cfi,label={[lab]above left:$a$}] at (a0) {};
        \node[cfi,label={[lab]above:$b$}] at (b0) {};
        \node[cfi,label={[lab]above:$c$}] at (c0) {};
        \node[cfi,label={[lab]above right:$a$}] at (a1) {};
        \node[cfi,label={[lab]below:$b$}] at (b1) {};
        \node[cfi,label={[lab]below:$c$}] at (c1) {};
        \node at (0,-1.55) {$\CFI_1(B)$};
      \else
        \treeA{a0}{180} \treeA{a1}{0}
        \treeB{b0}{180-\cfiAngle} \treeB{b1}{360-\cfiAngle}
        \treeC{c0}{\cfiAngle}     \treeC{c1}{180+\cfiAngle}
        \foreach \v in {a0,b0,c0,a1,b1,c1} \node[cfi] at (\v) {};
        \node at (0,-1.55) {$\widehat{\CFI}_1(B)$};
      \fi
    \end{scope}
  \end{scope}
}
\end{tikzpicture}
\caption{A CFI-compatible rooted product for $B=K_3$.
The hexagon labels indicate the base vertices of the fibres.
The attached paths $T_a,T_b,T_c$ have lengths one, two, and three;
each vertex in the fibre over $z$ receives a copy of $T_z$.}
\label{fig:ss-cfi-example}
\end{figure}

%% file: figures/cfi-blocks.tex
\begin{figure}[t]
\centering
\begin{tikzpicture}[
  x=.17cm,y=-.17cm,
  entry/.style={fill=black},
  gridline/.style={draw=black!15,line width=.1pt},
  fibreline/.style={draw=black!50,line width=.35pt},
  block/.style={draw=black,line width=.9pt},
  every node/.style={font=\scriptsize,inner sep=1pt},
]
\node[font=\small] at (14,-5) {$A=A_{\CFI_1(\Gamma^{\mathbf2})}$};
\foreach \v/\e/\k in {
  0/0/0,0/1/1,0/2/2,
  1/0/0,1/3/1,1/4/2,
  2/1/0,2/3/1,2/5/2,
  3/2/0,3/4/1,3/5/2
} {
  \ifnum\v=0
    \def\fibrebits{0/0/0/1,1/0/1/0,2/1/0/0,3/1/1/1}
  \else
    \def\fibrebits{0/0/0/0,1/0/1/1,2/1/0/1,3/1/1/0}
  \fi
  \foreach \r/\first/\second/\third in \fibrebits {
    \ifcase\k\edef\bit{\first}\or\edef\bit{\second}\or\edef\bit{\third}\fi
    \pgfmathtruncatemacro{\row}{4*\v+\r}
    \pgfmathtruncatemacro{\col}{16+2*\e+\bit}
    \fill[entry] (\col,\row) rectangle ++(1,1);
    \fill[entry] (\row,\col) rectangle ++(1,1);
  }
}
\foreach \position in {0,...,28} {
  \draw[gridline] (\position,0)--(\position,28);
  \draw[gridline] (0,\position)--(28,\position);
}
\foreach \boundary in {0,4,8,12,16,18,20,22,24,26,28} {
  \draw[fibreline] (\boundary,0)--(\boundary,28);
  \draw[fibreline] (0,\boundary)--(28,\boundary);
}
\foreach \position/\name in {
  2/a,6/b,10/c,14/d,17/ab,19/ac,21/ad,23/bc,25/bd,27/cd
} {
  \node[anchor=west,rotate=90] at (\position,-.5) {$\name$};
  \node[anchor=east] at (-.5,\position) {$\name$};
}
\node at (14,30.5) {Filled entry: $1$};

\draw[->,line width=.7pt] (30.5,14)--(39,14);
\node[align=center] at (34.75,18) {Fibre basis\\ordered by block};

\begin{scope}[xshift=7.1cm]
  \node[font=\small] at (14,-5) {$Q^{\top}AQ$};
  \foreach \e/\u/\v in {0/0/1,1/0/2,2/0/3,3/1/2,4/1/3,5/2/3} {
    \pgfmathtruncatemacro{\subdivision}{4+\e}
    \foreach \endpoint in {\u,\v} {
      \fill[entry] (\subdivision,\endpoint) rectangle ++(1,1);
      \fill[entry] (\endpoint,\subdivision) rectangle ++(1,1);
    }
    \pgfmathtruncatemacro{\start}{10+3*\e}
    \foreach \r/\c in {0/1,1/0,1/2,2/1} {
      \fill[entry] ({\start+\c},{\start+\r}) rectangle ++(1,1);
    }
  }
  \foreach \position in {0,...,28} {
    \draw[gridline] (\position,0)--(\position,28);
    \draw[gridline] (0,\position)--(28,\position);
  }
  \draw[fibreline] (0,0) rectangle (28,28);
  \draw[block] (0,0) rectangle (10,10);
  \node[anchor=west] at (29,5) {$A_{\Gamma^{\mathbf2}_0}$};
  \foreach \j/\e in {0/ab,1/ac,2/ad,3/bc,4/bd,5/cd} {
    \pgfmathtruncatemacro{\start}{10+3*\j}
    \draw[block] (\start,\start) rectangle ++(3,3);
    \node[anchor=west] at (29,{\start+1.5}) {$A_{\Gamma^{\mathbf2}_{\e}}$};
  }
  \node at (14,30.5) {Filled entry: $\sqrt2$};
\end{scope}
\end{tikzpicture}
\caption{Block decomposition for $\Gamma^{\mathbf2}$, where $\Gamma=K_4$
has vertices $a,b,c,d$, with attachments omitted.
The left matrix groups the $28$ CFI vertices by fibre and uses odd parity
at $a$. Within fibres, the bit vectors are in lexicographic order,
with incident-edge coordinates ordered by their labels.
The columns of $Q$ are the fibre basis vectors from the proof, ordered
first by the ten constant fibre vectors and then by the three signed fibre vectors
over $v,e,w$ for each edge $e=vw$ of $K_4$.
The right matrix has one $10\times10$ block $A_{\Gamma^{\mathbf2}_0}=\sqrt2 A_{\Gamma^{\mathbf2}}$ and
six $3\times3$ blocks $A_{\Gamma^{\mathbf2}_e}$; unfilled entries are zero.
Both CFI parities give the same right matrix.}
\label{fig:ss-cfi-blocks}
\end{figure}

%% file: characterisation.tex
\section{Proof of Theorem~\ref{thm:main}}
\label{sec:characterisation}

For a graph $G$, let
$W_G=\im A_G\subseteq\R^{V(G)}$.
For $d\in\{0,1,2,3\}$ and $f_1,\ldots,f_d\in\R^{V(G)}$, we define
the multilinear product
\begin{equation}\label{eq:image-trilinear}
 \langle f_1,\ldots,f_d\rangle_G
 =\sum_{x\in V(G)}\prod_{i=1}^{d}f_i(x).
\end{equation}
We define the empty product to be $1$, so $\langle\rangle_G=|V(G)|$.
The cases with one, two, and three arguments are the linear form,
the standard inner product, and the trilinear product, respectively.
We equip $W_G$ with the restricted inner product.

We call $G$ and $G'$ \emph{adjacency-image equivalent}, written
$G\sim_{\mathrm{im}}G'$, if $|V(G)|=|V(G')|$ and
there is a bijective linear map $U:W_G\to W_{G'}$ such that
\begin{equation}\label{eq:adjacency-image-equivalence}
 \begin{aligned}
 U(A_Gf)&=A_{G'}(Uf), &
 \langle Uf,Ug\rangle_{G'}&=\langle f,g\rangle_G, &
 \langle Uf,Ug,Uh\rangle_{G'}&=\langle f,g,h\rangle_G
 \end{aligned}
\end{equation}
for every $f,g,h\in W_G$.

We characterise $\equiv_{\mathcal D_3}$ by this equivalence
(\cref{sec:subcubic-characterisation}) and then show that it implies
isomorphism when one graph has nullity at most one
(\cref{sec:nullity-one-reconstruction}).

\subsection{Characterising subcubic homomorphism indistinguishability}
\label{sec:subcubic-characterisation}

Adjacency-image equivalence requires preservation of the inner and
trilinear products, but homomorphism counts from graphs with degree-one
vertices also involve the linear form.
The following lemma shows that the defining conditions already imply
its preservation.
\begin{lemma}\label{lem:image-linear-preservation}
Let $G\sim_{\mathrm{im}}G'$, and let $U:W_G\to W_{G'}$ witness this equivalence.
Then
\begin{equation}\label{eq:image-sum-transport}
 \langle Uf\rangle_{G'}=\langle f\rangle_G
\end{equation}
for every $f\in W_G$.
\end{lemma}
\begin{proof}
Let $(u_i)_{i\in I^*}$ be an orthonormal eigenbasis of $W_G$, indexed by
a finite set $I^*$, with $A_Gu_i=\lambda_i u_i$.
Then $(Uu_i)$ is an orthonormal eigenbasis of
$W_{G'}$, since $U$ is bijective, preserves the inner product, and
satisfies $A_{G'}Uu_i=UA_Gu_i=\lambda_i Uu_i$.

Let $\ones_G\in\R^{V(G)}$ be the all-ones vector.
For the coordinate vector $e_x$, we have $A_Ge_x\in W_G$ and
$\langle u_i,A_Ge_x\rangle_G=(A_Gu_i)(x)$ by symmetry.
Parseval's identity therefore gives
$\sum_i((A_Gu_i)(x))^2=\|A_Ge_x\|^2=A_G^2(x,x)$.
Since $G$ is simple and undirected,
$A_G^2(x,x)=\deg_G(x)=(A_G\ones_G)(x)$.
Multiplying by $h(x)$ and summing over $x$ gives, for every $h\in W_G$,
\begin{equation}\label{eq:image-linear}
 \sum_i\langle h,A_Gu_i,A_Gu_i\rangle_G
 =\langle h,A_G\ones_G\rangle_G
 =\langle A_Gh\rangle_G.
\end{equation}
\Cref{eq:image-linear} also holds for $G'$ with the eigenbasis
$(Uu_i)$.
Now let $f\in W_G$. Since $A_G|_{W_G}$ has only nonzero eigenvalues,
it is invertible, so $f=A_Gg$ for some $g\in W_G$.
Hence, since $A_{G'}U=UA_G$ and $U$ preserves the trilinear
product,
\begin{equation}\label{eq:commentary-linear-transport}
 \begin{aligned}
 \langle Uf\rangle_{G'}
 &=\langle A_{G'}Ug\rangle_{G'}
 =\sum_i\langle Ug,A_{G'}Uu_i,A_{G'}Uu_i\rangle_{G'}\\
 &=\sum_i\langle Ug,UA_Gu_i,UA_Gu_i\rangle_{G'}
 =\sum_i\langle g,A_Gu_i,A_Gu_i\rangle_G
 =\langle A_Gg\rangle_G
 =\langle f\rangle_G,
 \end{aligned}
\end{equation}
which is \cref{eq:image-sum-transport}.
\end{proof}

Subdivision counts have an expansion in the spectral projectors of the target graph.
We use finite multigraphs, allowing loops and parallel edges; a loop
contributes two to the degree. For a multigraph $H$ and a tuple
$\mathbf m=(m_e)_{e\in E(H)}$ of positive integers, $H^{\mathbf m}$
is the subdivision that replaces each edge $e$ by a path of length $m_e$,
with the endpoints identified when $e$ is a loop.
In particular, $\Gamma^{\mathbf2}$ is the case $\mathbf m=(2,\ldots,2)$.
A subdividing path with $q_e$ internal vertices has length $m_e=q_e+1$.
When every $m_e\geq3$, this subdivision is simple, even if $H$ has loops
or parallel edges.

For a graph $G$, let $\Lambda_G$ be the set of
nonzero eigenvalues of $A_G$.
For $\lambda\in\Lambda_G$, $E_\lambda$ denotes the
$\lambda$-eigenspace and $P_\lambda$ its orthogonal projector.
Since we retain only the spectral data on $W_G$, we set
$E_\lambda=\{0\}$ and $P_\lambda=0$ for $\lambda\notin\Lambda_G$,
including $\lambda=0$.
We add a graph subscript, as in $E_{\lambda,G}$ and $P_{\lambda,G}$,
when needed. For a multigraph $H$ and an edge labelling
$\ell:E(H)\to\R^\times$, the \emph{labelled coefficient} is
defined by
\begin{equation}\label{eq:contraction}
 c_G(H,\ell)=\sum_{\varphi:V(H)\to V(G)}
    \prod_{e=vw\in E(H)}P_{\ell(e)}(\varphi(v),\varphi(w)).
\end{equation}
For a path of length $m_e$ with endpoint images fixed at $x,y$,
summing over the images of its internal vertices gives
$A_G^{m_e}(x,y)=\sum_{\lambda\in\Lambda_G}\lambda^{m_e}P_\lambda(x,y)$.
For lengths $m_e\geq3$, multiplying out these
expansions over all edges and summing over the original vertex images gives
\begin{equation}\label{eq:expansion}
 \hom(H^{\mathbf m},G)=\sum_{\ell:E(H)\to\Lambda_G}
       c_G(H,\ell)\prod_{e\in E(H)}\ell(e)^{m_e}.
\end{equation}
\begin{lemma}\label{lem:subdivision-coefficients}
For two graphs $G$ and $G'$ and a multigraph $H$,
$\hom(H^{\mathbf m},G)=\hom(H^{\mathbf m},G')$ for every tuple
$\mathbf m$ with $m_e\geq3$ for all $e$ if and only if $c_G(H,\ell)=c_{G'}(H,\ell)$
for every edge labelling $\ell:E(H)\to\R^\times$.
\end{lemma}
\begin{proof}
Equality of the labelled coefficients gives equality of the counts by
\cref{eq:expansion}: since the coefficients vanish whenever a label
lies outside the target's nonzero spectrum, both sums may be taken over
$\ell:E(H)\to\Lambda_G\cup\Lambda_{G'}$.
Conversely, suppose the counts agree for all these subdivisions.
The functions $m\mapsto\lambda^m$ for
$\lambda\in\Lambda_G\cup\Lambda_{G'}$ are linearly independent on the
integers $m\geq3$, by the Vandermonde determinant.
Comparing the coefficients of $\prod_e\ell(e)^{m_e}$ in
\cref{eq:expansion} for $G$ and for $G'$, one variable $m_e\geq3$ at a time,
gives
$c_G(H,\ell)=c_{G'}(H,\ell)$ for every $\ell$.
If a label lies outside both spectra, both coefficients vanish.
\end{proof}

We obtain adjacency-image equivalence from equality of subcubic
homomorphism counts using tensor networks.
A \emph{tensor} on a Euclidean space is a multilinear form, and its
\emph{coefficients} in a basis are its values on tuples of basis
vectors.
A \emph{closed tensor network} is a finite multigraph with a tensor at
each vertex and one tensor index for each incident edge.
Loops and parallel edges are allowed; a loop uses two indices at its vertex.
The value of the network is its contraction: in a fixed orthonormal basis,
we multiply tensor coefficients and sum over the indices paired along edges.
We use the following result, called the converse of the real orthogonal
Holant theorem.
\begin{theorem*}[{\cite[Theorem~9]{Young2025}, specialised to tensor networks with all degrees two or three}]
Let $G$ and $G'$ be graphs with $\dim W_G=\dim W_{G'}>0$.
Fix orthonormal eigenbases of $W_G$ and $W_{G'}$ for the network evaluations.
For a finite multigraph $\mathcal N$ whose vertices all have degree two or
three, let $\val_G(\mathcal N)$ be the value of the closed tensor network
on $W_G$ with underlying multigraph $\mathcal N$ that assigns the adjacency form
$(f,g)\mapsto\langle f,A_Gg\rangle_G$ to each degree-two vertex and
the trilinear product $\langle\cdot,\cdot,\cdot\rangle_G$ to each
degree-three vertex. Let $\val_{G'}(\mathcal N)$ be the value of the same
construction with $G'$ in place of $G$.
Then the equality
\begin{equation}\label{eq:image-all-networks}
 \val_G(\mathcal N)=\val_{G'}(\mathcal N)
\end{equation}
holds for every such multigraph $\mathcal N$ if and only if there is a bijective
linear map $U:W_G\to W_{G'}$ that intertwines the adjacency operators and
preserves both the inner products and the trilinear products.
\end{theorem*}

The labelled coefficients connect subdivision counts with tensor-network
evaluations, providing the link between subcubic indistinguishability and
adjacency-image equivalence.

\begin{lemma}\label{lem:characterisation}
 For any finite simple graphs $G$ and $G'$, the following are equivalent:
 \begin{enumerate}
 \item[(a)] $G\equiv_{\mathcal D_3}G'$;
 \item[(b)] $G\sim_{\mathrm{im}}G'$;
 \item[(c)] $|V(G)|=|V(G')|$ and $c_G(H,\ell)=c_{G'}(H,\ell)$ for every
 connected cubic multigraph $H$ and every edge labelling
 $\ell:E(H)\to\R^\times$.
 \end{enumerate}
\end{lemma}
\begin{proof}
We first prove (b)$\Rightarrow$(a).
Suppose $G\sim_{\mathrm{im}}G'$, witnessed by $U$.
Let $(u_i)_{i\in I^*}$ be an orthonormal eigenbasis of $W_G$, indexed by
a finite set $I^*$, with $A_Gu_i=\lambda_i u_i$.
Then $(Uu_i)$ is an orthonormal eigenbasis of
$W_{G'}$, since $U$ is bijective, preserves the inner product, and
satisfies $A_{G'}Uu_i=UA_Gu_i=\lambda_i Uu_i$.

Fix a subcubic graph $F$. Expanding each adjacency factor in the
homomorphism count as
$A_G(x,y)=\sum_{i\in I^*}\lambda_i u_i(x)u_i(y)$
and summing first over the images of the vertices of $F$ gives
\begin{equation}\label{eq:image-hom-expansion}
 \hom(F,G)=\sum_{\eta:E(F)\to I^*}
 \left(\prod_{e\in E(F)}\lambda_{\eta(e)}\right)
 \prod_{v\in V(F)}\bigl\langle u_{\eta(e)}:e\ni v\bigr\rangle_G.
\end{equation}
The arguments of $\langle u_{\eta(e)}:e\ni v\rangle_G$ are the vectors
indexed by the edge incidences at $v$.
At an isolated vertex the factor is $\langle\rangle_G=|V(G)|$.
The expansion for $G'$ in the basis $(Uu_i)_{i\in I^*}$ has the same
eigenvalue factors. The corresponding vertex factors agree by the
assumption $G\sim_{\mathrm{im}}G'$ for degrees $0,2,3$, and by
\cref{lem:image-linear-preservation} for degree $1$.
Hence $\hom(F,G)=\hom(F,G')$.

We next prove (a)$\Rightarrow$(c).
Suppose $G\equiv_{\mathcal D_3}G'$. The count from $K_1$ gives
$|V(G)|=|V(G')|$. Every simple subdivision of a connected cubic multigraph $H$ is subcubic,
so its homomorphism counts into $G$ and $G'$ agree.
\Cref{lem:subdivision-coefficients} gives
$c_G(H,\ell)=c_{G'}(H,\ell)$ for every nonzero edge labelling $\ell$.

We now prove (c)$\Rightarrow$(b). Suppose (c) holds.
We first show that $G$ and $G'$ have the same nonzero eigenvalues with
the same multiplicities. Let $\Theta$ be the cubic multigraph with
vertices $a,b$ and three parallel edges. We write
$c_G(\Theta;\lambda,\mu,\nu)$ for the coefficient obtained by assigning
the labels $\lambda,\mu,\nu$ to its three edges.
For every $\lambda\in\Lambda_G\cup\Lambda_{G'}$,
\begin{multline}\label{eq:theta-multiplicity}
 \sum_{\mu,\nu\in\Lambda_G\cup\Lambda_{G'}}\mu\nu c_G(\Theta;\lambda,\mu,\nu)
 =\sum_{x,y\in V(G)}P_{\lambda,G}(x,y)A_G(x,y)^2\\
 =\sum_{x,y\in V(G)}P_{\lambda,G}(x,y)A_G(x,y)
 =\tr(P_{\lambda,G}A_G)
 =\lambda\dim E_{\lambda,G}.
\end{multline}
The first equality follows from \cref{eq:contraction} and
$A_G=\sum_{\mu\in\Lambda_G\cup\Lambda_{G'}}\mu P_{\mu,G}$; the second uses that $G$ is
simple, so every adjacency entry is $0$ or $1$.
The same identity holds for $G'$. The assumed equality of the labelled coefficients equates the left-hand sides,
so division by $\lambda\ne0$ gives
$\dim E_{\lambda,G}=\dim E_{\lambda,G'}$.
Every nonzero real number outside $\Lambda_G\cup\Lambda_{G'}$ has multiplicity zero for
both graphs.
The adjacency images $W_G$ and $W_{G'}$ therefore have the same dimension.
If this dimension is zero, the unique linear map $W_G\to W_{G'}$ witnesses $G\sim_{\mathrm{im}}G'$.

Otherwise, since $\dim W_G=\dim W_{G'}$ and $|V(G)|=|V(G')|$, \cite[Theorem~9]{Young2025}, specialised to networks with all degrees two or three,
gives $G\sim_{\mathrm{im}}G'$ if \cref{eq:image-all-networks} holds
for every such network $\mathcal N$.
Since network values multiply over connected components, it is enough
to prove $\val_G(\mathcal N)=\val_{G'}(\mathcal N)$ for every connected $\mathcal N$.

Fix a connected multigraph $\mathcal N$ with all degrees two or three. Either every vertex of
$\mathcal N$ has degree two, or $\mathcal N$ has a vertex of degree three.
If every vertex of $\mathcal N$ has degree two, then $\mathcal N$ is a cycle.
The cyclic contraction of its $|V(\mathcal N)|$ adjacency forms gives
\begin{equation}\label{eq:degree-two-trace}
 \val_G(\mathcal N)=\tr\bigl((A_G|_{W_G})^{|V(\mathcal N)|}\bigr).
\end{equation}
Since $A_G|_{W_G}$ and $A_{G'}|_{W_{G'}}$ have the same spectrum,
their $|V(\mathcal N)|$th powers have equal traces. Thus $\val_G(\mathcal N)=\val_{G'}(\mathcal N)$.

Now suppose $\mathcal N$ has a vertex of degree three.
A degree-two vertex cannot carry a loop, since it would form an entire
connected component separate from the degree-three vertex, contradicting connectedness.
Thus each degree-two vertex has two non-loop incident edges.
Let $\overline{\mathcal N}$ be the connected cubic multigraph obtained from $\mathcal N$ by repeatedly
removing a degree-two vertex and replacing its two incident edges with
a single edge between its two neighbours, or a loop if they coincide.
Each edge $e$ of $\overline{\mathcal N}$ replaces a path in $\mathcal N$,
or a cycle when $e$ is a loop; let $q_e\geq0$ be the number of degree-two
vertices on that path or cycle.
\Cref{fig:network-subdivision} illustrates the network on $\mathcal N$, its cubic multigraph,
and a simple subdivision.
\input{figures/network-subdivision}
We claim that
\begin{equation}\label{eq:image-network-expansion}
 \val_G(\mathcal N)=\sum_{\ell:E(\overline{\mathcal N})\to\Lambda_G}
 \left(\prod_{e\in E(\overline{\mathcal N})}\ell(e)^{q_e}\right)c_G(\overline{\mathcal N},\ell).
\end{equation}
Contracting the tensors at the $q_e$ degree-two vertices of the path or cycle
represented by $e$ gives the bilinear form
$(f,g)\mapsto\langle f,(A_G|_{W_G})^{q_e}g\rangle_G$ on $W_G$,
including the inner product when $q_e=0$.
We evaluate the remaining contraction in the fixed orthonormal eigenbasis
$(u_i)_{i\in I^*}$ of $W_G$, with $A_Gu_i=\lambda_i u_i$.
The bilinear form for $e$ has coefficients $\lambda_i^{q_e}\delta_{ij}$.
Summing over the common eigenbasis index on each edge gives
\begin{equation}\label{eq:image-network-index-expansion}
 \val_G(\mathcal N)=\sum_{\eta:E(\overline{\mathcal N})\to I^*}
 \left(\prod_e\lambda_{\eta(e)}^{q_e}\right)
 \prod_{v\in V(\overline{\mathcal N})}
 \bigl\langle u_{\eta(e)}:e\ni v\bigr\rangle_G,
\end{equation}
where a loop contributes two arguments to its vertex factor.
Expanding each remaining trilinear coefficient as
$\langle u_i,u_j,u_k\rangle_G=\sum_xu_i(x)u_j(x)u_k(x)$ and
interchanging the sums in \cref{eq:image-network-index-expansion} gives
\begin{equation}\label{eq:image-network-vertex-expansion}
 \val_G(\mathcal N)=\sum_{\varphi:V(\overline{\mathcal N})\to V(G)}
 \prod_{e=vw\in E(\overline{\mathcal N})}
 \left(\sum_i\lambda_i^{q_e}u_i(\varphi(v))u_i(\varphi(w))\right).
\end{equation}
Grouping each inner sum by its eigenvalue gives
$\sum_{\lambda\in\Lambda_G}\lambda^{q_e}P_{\lambda,G}(\varphi(v),\varphi(w))$.
Multiplying out the edge factors now gives
\cref{eq:image-network-expansion}, by \cref{eq:contraction}.

By (c), $c_G(\overline{\mathcal N},\ell)=c_{G'}(\overline{\mathcal N},\ell)$ for every
nonzero edge labelling $\ell$.
Since $\Lambda_G=\Lambda_{G'}$, \cref{eq:image-network-expansion} gives
\begin{equation}\label{eq:image-network-equality}
 \val_G(\mathcal N)
 =\sum_{\ell:E(\overline{\mathcal N})\to\Lambda_G}
 \left(\prod_e\ell(e)^{q_e}\right)c_{G'}(\overline{\mathcal N},\ell)
 =\val_{G'}(\mathcal N).
\end{equation}
Thus \cref{eq:image-all-networks} holds. By \cite[Theorem~9]{Young2025} there
is a bijective linear map $U$ that intertwines the adjacency operators and
preserves both the inner products and the trilinear products;
since also $|V(G)|=|V(G')|$, $G\sim_{\mathrm{im}}G'$.
\end{proof}

\subsection{Reconstructing graphs of nullity at most one}
\label{sec:nullity-one-reconstruction}

We reconstruct graphs by comparing their vertex rows in orthonormal eigenbases.
For a graph $G$, we extend an orthonormal eigenbasis of $W_G$ to an
orthonormal eigenbasis $(u_i)_{i\in I}$ of $\R^{V(G)}$.
Let $\lambda_i$ be the eigenvalue of $u_i$ for each $i\in I$,
and let $I^{*}=\{i\in I:\lambda_i\ne0\}$.
Let $\gamma_1$ and $\gamma_3$ be the coefficient arrays of the
linear form and trilinear product in this eigenbasis, given by
\begin{equation}
  \label{eq:tensors}
  \begin{aligned}
  \gamma_1(i)&=\langle u_i\rangle_G, &
  \gamma_3(i,j,k)&=\langle u_i,u_j,u_k\rangle_G
  \end{aligned}
\end{equation}
for $i,j,k\in I$. For $x\in V(G)$, let $r_x=(u_i(x))_{i\in I}\in\R^I$ be
the row of $x$. These rows form an orthonormal basis of $\R^I$.
For $r\in\R^I$, let $r^{\otimes3}$ denote the array with entries
$r(i)r(j)r(k)$ for $i,j,k\in I$.
The rows satisfy
\begin{equation}
  \label{eq:rows}
  \begin{aligned}
  \gamma_3 &= \sum_{x \in V(G)} r_x^{\otimes 3}, &
  A_G(x,y) &= \sum_{i \in I} \lambda_i r_x(i) r_y(i).
  \end{aligned}
\end{equation}
We write $\gamma_3^{*}$ for the restriction of $\gamma_3$ to
indices in $I^{*}$.
In the eigenbasis $(u_i)_{i\in I^*}$ of $W_G$, the trilinear product
$\langle\cdot,\cdot,\cdot\rangle_G$ has coefficient array $\gamma_3^*$.
When $G$ has nullity one, we index the basis so that $I=I^*\sqcup\{0\}$
and $u_0$ spans $\ker A_G$.
In particular, $\gamma_1(0)=\sum_{x\in V(G)}u_0(x)$.
The coefficients of $\gamma_3$ outside
$\gamma_3^*$ are the $\gamma_3(i,j,k)$ with at least one
index equal to $0$.
For such a graph, let $\tilde{r}_x=(u_i(x))_{i\in I^*}$ be
its reduced row at $x$.
Since $\tilde{r}_x$ is $r_x$ with its kernel coordinate omitted,
orthonormality of the rows gives
\begin{equation}\label{eq:gram}
 \langle\tilde{r}_x,\tilde{r}_y\rangle=\delta_{xy}-u_0(x)u_0(y).
\end{equation}

We will use the following neighbourhood constraint in the reconstruction argument.
Suppose $G$ has nullity one and $|\supp u_0|>2$.
For distinct $x,y\in\supp u_0$, let $e_x,e_y\in\R^{V(G)}$ be the
standard coordinate vectors.
The vector $e_x-e_y$ has support of size two, so it is not a scalar multiple of
$u_0$. Thus $A_G(e_x-e_y)\ne0$, which gives
\begin{equation}\label{eq:support-neighbourhoods}
 N_G(x)\ne N_G(y)\quad\text{for distinct }x,y\in\supp u_0.
\end{equation}

The following lemma generalises the fact that two orthonormal bases with
the same cubic tensor agree up to permutation~\cite[Theorem~3.2]{ZhangGolub2001}.
When the corresponding eigenvalues agree, such a permutation of the
vertex rows induces a graph isomorphism by \cref{eq:rows}.
The lemma also identifies unit reduced rows from $\gamma_3^*$.
\begin{lemma}\label{lem:odeco}
 Let $X,I$ be finite sets, and let $R\in\R^{X\times I}$ be a matrix
 with orthonormal columns. Write $r_x\in\R^I$ for its row indexed by
 $x\in X$, and define $\kappa:\R^I\to\R$ by
 $\kappa(\xi)=\sum_{x\in X}\langle r_x,\xi\rangle^3$.
 Then
 \begin{equation}\label{eq:cubic-unit-rows}
  \{\xi:\lVert \xi\rVert=1,\ \kappa(\xi)=1\}
  =\{r_x:\lVert r_x\rVert=1\}.
 \end{equation}
 In particular, if $(r_x)_{x\in X}$ and $(r'_y)_{y\in Y}$ are
 orthonormal bases of the same Euclidean space and
 $\sum_x r_x^{\otimes3}=\sum_y(r'_y)^{\otimes3}$, then there is a
 bijection $\pi:X\to Y$ with $r'_{\pi(x)}=r_x$ for every $x\in X$.
\end{lemma}
\begin{proof}
 Let $\xi$ be a unit vector and put $c=R\xi$.
 Since the columns of $R$ are orthonormal, we have $\sum_xc_x^2=1$,
 so $c_x\leq1$ for every $x$. Hence
 \begin{equation}\label{eq:cubic-unit-defect}
  1-\kappa(\xi)=\sum_x c_x^2(1-c_x)\geq0.
 \end{equation}
 Equality holds exactly when $c=e_x$ for some $x$,
 where $e_x$ is the standard coordinate vector.
 Then $\xi=R^{\top}c=r_x$.
 Conversely, if $\lVert r_x\rVert=1$, then $Rr_x$ is a unit vector
 whose $x$-coordinate is $1$, so $Rr_x=e_x$ and
 $\kappa(r_x)=1$. This proves \cref{eq:cubic-unit-rows}.
 The final assertion follows because equal cubic tensors identify the
 same set of unit rows.
\end{proof}

We now prove \cref{lem:reconstruction}, which converts adjacency-image
equivalence into graph isomorphism under the nullity bound.
\Cref{thm:main} then follows from the characterisation in
\cref{lem:characterisation}.

\begin{lemma}\label{lem:reconstruction}
 Let $G$ have nullity at most one, and let $G'$ be a graph.
 If $G\sim_{\mathrm{im}}G'$, then $G\cong G'$.
\end{lemma}
\begin{proof}
  We fix a bijective linear map $U:W_G\to W_{G'}$ witnessing
  $G\sim_{\mathrm{im}}G'$.
  The graphs $G$ and $G'$ have the same number of vertices because
  $G\sim_{\mathrm{im}}G'$. Since $U$ is bijective and the graphs have the
  same number of vertices, the rank-nullity theorem gives equal nullities,
  at most one.
  We fix an orthonormal eigenbasis
  $(u_i)_{i\in I}$ of $\R^{V(G)}$.
  We set $u'_i=Uu_i$ for $i\in I^*$.
  By \cref{eq:adjacency-image-equivalence}, these vectors form an orthonormal
  eigenbasis of $W_{G'}$. We extend it by an orthonormal basis of $\ker A_{G'}$.
  The resulting family $(u'_i)_{i\in I}$ is an orthonormal eigenbasis of $\R^{V(G')}$.
  Then $u_i$ and $u'_i$ have the same eigenvalue $\lambda_i$ for every
  $i\in I$. We use primes for quantities associated with $G'$.
  By \cref{eq:adjacency-image-equivalence} and
  \cref{lem:image-linear-preservation}, for $i,j,k\in I^*$,
  \begin{equation}\label{eq:transported-tensors}
    \begin{aligned}
      \gamma'_1(i)&=\langle Uu_i\rangle_{G'}
        =\langle u_i\rangle_G=\gamma_1(i),\\
      \gamma'_3(i,j,k)&=\langle Uu_i,Uu_j,Uu_k\rangle_{G'}
        =\langle u_i,u_j,u_k\rangle_G=\gamma_3(i,j,k).
    \end{aligned}
  \end{equation}
  For $x\in V(G)$ and $y\in V(G')$, let
  $\tilde{r}_x=(u_i(x))_{i\in I^*}$ and
  $\tilde{r}'_y=(u'_i(y))_{i\in I^*}$ be the reduced rows.
  A bijection $\pi:V(G)\to V(G')$ with $\tilde{r}'_{\pi(x)}=\tilde{r}_x$ is a
  graph isomorphism. The zero-eigenvalue terms in the second identity of
  \cref{eq:rows} vanish, so for $x,y\in V(G)$ we have
  \begin{equation}\label{eq:row-isomorphism}
    A_{G'}(\pi(x),\pi(y))
    =\sum_{i\in I^*}\lambda_i\tilde{r}'_{\pi(x)}(i)\tilde{r}'_{\pi(y)}(i)
    =\sum_{i\in I^*}\lambda_i\tilde{r}_x(i)\tilde{r}_y(i)=A_G(x,y).
  \end{equation}
  If $\gamma_3=\gamma'_3$, then by \cref{lem:odeco} there is a
  bijection $\pi$ with $r'_{\pi(x)}=r_x$, hence $\tilde{r}'_{\pi(x)}=\tilde{r}_x$.
  If the graphs are nonsingular, then $I^*=I$, so
  \cref{eq:transported-tensors} already gives $\gamma_3=\gamma'_3$
  and hence $G\cong G'$.

  We therefore assume that both graphs have nullity one.
  We index $(u_i)$ and $(u'_i)$ so that $I=I^*\sqcup\{0\}$, with unit kernel
  vectors $u_0,u'_0$.
  We first compare the coefficients involving index $0$, which are
  not compared in \cref{eq:transported-tensors}.
  The \emph{missing coefficients} are the linear coefficient $\gamma_1(0)$
  and, up to symmetry, the cubic coefficients $\gamma_3(0,0,0)$, $\gamma_3(0,0,k)$, and
  $\gamma_3(0,j,k)$ for $j,k\in I^*$.
  Once $(u_i)_{i\in I^*}$ and $(u'_i)_{i\in I^*}$ are fixed,
  the only remaining choices are the signs of $u_0$ and $u'_0$.
  If the signs of $u_0,u'_0$ can be chosen so that these coefficients
  agree with their primed counterparts, then \cref{eq:transported-tensors}
  gives $\gamma_3=\gamma'_3$, hence $G\cong G'$ as above.
  Replacing $u_0$ by $-u_0$ negates $\gamma_1(0)$ and
  $\gamma_3(0,j,k)$. Replacing $u'_0$ by $-u'_0$
  has the same effect on the primed coefficients.

  Orthonormality of the full eigenbasis gives, for $x,y\in V(G)$ and $j,k\in I$,
  \begin{equation}\label{eq:full-basis-orthogonality}
    \begin{aligned}
      \sum_{m\in I}u_m(x)u_m(y)&=\delta_{xy}, &
      \sum_{x\in V(G)}u_j(x)u_k(x)&=\delta_{jk}.
    \end{aligned}
  \end{equation}
  Substituting the definitions in \cref{eq:tensors} and using
  \cref{eq:full-basis-orthogonality} gives the following identities
  for all $i,j,k,l\in I$:
  \begin{gather}
    \label{eq:id-norm}
    \sum_{m \in I} \gamma_1(m)^{2} = |V(G)|, \\
    \label{eq:id-unit}
    \sum_{m \in I} \gamma_1(m) \gamma_3(m,j,k) = \delta_{jk}, \\
    \label{eq:id-trace}
    \sum_{m \in I} \gamma_3(j,m,m) = \gamma_1(j), \\
    \label{eq:id-assoc}
    \sum_{m \in I} \gamma_3(i,j,m) \gamma_3(m,k,l)
      = \sum_{x \in V(G)} u_i(x) u_j(x) u_k(x) u_l(x).
  \end{gather}
  Separating the term $m=0$ in \cref{eq:id-trace} gives
  \begin{equation}
    \label{eq:trace-completion}
    \begin{aligned}
      \gamma_3(0,0,k) &= \gamma_1(k) - \sum_{i\in I^*}\gamma_3(k,i,i)
        &&\text{for } k\in I^*,\\
      \gamma_3(0,0,0) &= \gamma_1(0) - \sum_{i\in I^*}\gamma_3(0,i,i).
    \end{aligned}
  \end{equation}
  The first identity involves only coefficients at indices in $I^*$ on
  the right-hand side, so $\gamma'_3(0,0,k)=\gamma_3(0,0,k)$ by
  \cref{eq:transported-tensors}. The second determines
  $\gamma_3(0,0,0)$ from $\gamma_1(0)$ and the coefficients
  $\gamma_3(0,j,k)$. Thus agreement of all missing coefficients
  reduces to agreement of $\gamma_1(0)$ and $\gamma_3(0,j,k)$
  with their primed counterparts.

  For the case analysis, let $\mathsf C_0$ be the symmetric
  $I^*\times I^*$ matrix defined by
  \begin{equation}\label{eq:unknowns}
    \mathsf C_0(j,k) = \gamma_3(0,j,k).
  \end{equation}
  Separating the term $m=0$ in \cref{eq:id-norm} gives
  $\gamma_1(0)^2=|V(G)|-\sum_{i\in I^*}\gamma_1(i)^2$, so
  $\gamma_1(0)^2=(\gamma'_1(0))^2$ by \cref{eq:transported-tensors}.
  We divide the remaining proof into the cases $\gamma_1(0)\ne0$ (Case A)
  and $\gamma_1(0)=0$ (Case B).
  Within Case B, we distinguish $\rank\mathsf C_0\geq3$ (Case B1)
  from $\rank\mathsf C_0\leq2$ (Case B2).
  In Cases A and B1, we show that a choice of sign for $u'_0$ gives
  $\gamma'_1(0)=\gamma_1(0)$ and $\mathsf C'_0=\mathsf C_0$.
  In Case B2, we use the graph structure to show that $u_0$ has
  support of size two or four. In both subcases, we show directly
  that $(\tilde{r}_x)_{x\in V(G)}$ and $(\tilde{r}'_y)_{y\in V(G')}$ agree as multisets.

  \emph{Case A: $\gamma_1(0) \neq 0$.} We choose the sign of $u'_0$ so
  that $\gamma'_1(0)=\gamma_1(0)$. Separating the term $m=0$ in
  \cref{eq:id-unit} gives
  \begin{equation}
    \label{eq:caseA}
    \gamma_1(0) \mathsf{C}_0(j,k) = \delta_{jk} - \sum_{i \in I^{*}} \gamma_1(i) \gamma_3(i,j,k).
  \end{equation}
  The right-hand side for $G'$ equals the right-hand side for $G$ by
  \cref{eq:transported-tensors}, and dividing by
  $\gamma_1(0)=\gamma'_1(0)\neq0$ gives $\mathsf C'_0=\mathsf C_0$.
  Hence $\gamma_3=\gamma'_3$.

  \emph{Case B: $\gamma_1(0) = 0$.} Then $\gamma'_1(0)=0$, since
  $(\gamma'_1(0))^2=\gamma_1(0)^2$. Here \cref{eq:id-unit} no longer
  determines $\mathsf C_0$, so we distinguish two cases according to
  the rank of $\mathsf C_0$.

  \emph{Case B1: $\rank \mathsf{C}_0 \geq 3$.} We show that
  $\mathsf C'_0=\pm\mathsf{C}_0$ by comparing their $2\times2$ minors.
  For $i,j,k,l \in I^{*}$, let $Q(i,j;k,l)$ be the sum of the terms with
  $m\neq0$ in \cref{eq:id-assoc},
  \begin{equation}\label{eq:q-tensor}
   Q(i,j;k,l) = \sum_{m \in I^{*}} \gamma_3(i,j,m) \gamma_3(m,k,l),
  \end{equation}
  which involves only $\gamma_3^*$.
  Separating the term $m=0$ in \cref{eq:id-assoc} gives
  \begin{equation}
    \label{eq:assoc-split}
    \mathsf{C}_0(i,j)\mathsf{C}_0(k,l)+Q(i,j;k,l)
    =\sum_{x \in V(G)} u_i(x) u_j(x) u_k(x) u_l(x).
  \end{equation}
  The right-hand side is symmetric in $(i,j,k,l)$, so subtracting from
  \cref{eq:assoc-split} the same identity with $j$ and $k$ exchanged
  gives
  \begin{equation}
    \label{eq:minors}
    \mathsf{C}_0(i,j) \mathsf{C}_0(k,l) - \mathsf{C}_0(i,k) \mathsf{C}_0(j,l) = Q(i,k;j,l) - Q(i,j;k,l).
  \end{equation}
  The left-hand side is the minor of $\mathsf{C}_0$ with rows $i,l$ and
  columns $j,k$. The right-hand side involves only $\gamma_3^*$,
  which $G$ and $G'$ share by \cref{eq:transported-tensors}.
  Thus $\mathsf C_0$ and $\mathsf C'_0$ have equal corresponding $2\times2$
  minors, hence equal exterior squares.
  Since $\rank\mathsf C_0\geq3$, \cite[Corollary~4.1]{Greenhill1999}
  gives $\mathsf C'_0=\pm\mathsf C_0$.
  We choose the sign of $u'_0$ so that $\mathsf C'_0=\mathsf{C}_0$.
  Hence $\gamma_3=\gamma'_3$.

  \emph{Case B2: $\rank \mathsf{C}_0 \leq 2$.} We show that
  $\{\tilde{r}_z:z\notin\supp u_0\}=\{\tilde{r}'_z:z\notin\supp u'_0\}$.
  Let $\tilde R$ be the matrix with rows $\tilde{r}_x$. Its columns are
  orthonormal, and \cref{eq:gram} gives
  $\lVert\tilde{r}_x\rVert^2=1-u_0(x)^2$.
  Thus $\tilde{r}_x$ is a unit vector exactly when $x\notin\supp u_0$.
  The same holds for $\tilde{r}'_y$ and $\supp u'_0$.
  Let $\kappa$ and $\kappa'$ be the cubic forms associated with the coefficient
  arrays $\gamma_3^*$ and $(\gamma'_3)^*$, respectively, so
  \begin{equation}\label{eq:reduced-cubic}
   \kappa(\xi)=\sum_x\langle\tilde{r}_x,\xi\rangle^3
   =\sum_y\langle\tilde{r}'_y,\xi\rangle^3=\kappa'(\xi)
  \end{equation}
  by \cref{eq:transported-tensors}.
  By \cref{lem:odeco} and \cref{eq:reduced-cubic},
  \begin{equation}\label{eq:outside-rows}
   \{\tilde{r}_z:z\notin\supp u_0\}
   =\{\xi:\lVert \xi\rVert=1,\ \kappa(\xi)=1\}
   =\{\tilde{r}'_z:z\notin\supp u'_0\}.
  \end{equation}
  Each of these row families is orthonormal by \cref{eq:gram}, so its
  members are distinct. Since $G$ and $G'$ have the same number of
  vertices, \cref{eq:outside-rows} also gives
  $|\supp u'_0|=|\supp u_0|$.
  We relabel the vertices of $G'$ outside $\supp u'_0$ so that
  $\tilde{r}'_z=\tilde{r}_z$ for every $z\notin\supp u_0$.

  It remains to compare $(\tilde{r}_x)_{x\in\supp u_0}$ and
  $(\tilde{r}'_y)_{y\in\supp u'_0}$ as multisets.
  Their sums agree by
  \cref{eq:transported-tensors} and \cref{eq:outside-rows},
  \begin{equation}\label{eq:support-row-sum}
    \sum_{x\in\supp u_0}\tilde{r}_x
    = (\gamma_1(i))_{i\in I^*}-\sum_{z\notin\supp u_0}\tilde{r}_z
    = \sum_{y\in\supp u'_0}\tilde{r}'_y.
  \end{equation}
  We claim that $|\supp u_0|\in\{2,4\}$.
  Since $\mathsf{C}_0 = \tilde R^{\top} \diag(u_0) \tilde R$ and
  $\rank\tilde R=|V(G)|-1$, two applications of Sylvester's rank inequality give
  \begin{equation}\label{eq:support-rank-bound}
    \begin{aligned}
      \rank\mathsf C_0
      &\geq\rank(\diag(u_0)\tilde R)-1\\
      &\geq\rank\diag(u_0)-2=|\supp u_0|-2.
    \end{aligned}
  \end{equation}
  Hence
  $|\supp u_0| \leq \rank \mathsf{C}_0 + 2 \leq 4$.
  The support has at least two vertices because $u_0\ne0$ and
  $\gamma_1(0)=0$.
  Suppose it has three vertices. For every nonempty proper subset
  $T$ of $\supp u_0$, we have $\sum_{x\in T}u_0(x)\ne0$, since
  the three entries are nonzero and their sum is zero.
  The equation $(A_Gu_0)(w)=\sum_{x\in N_G(w)}u_0(x)=0$ therefore
  forces every vertex $w$ to be adjacent to all three support
  vertices or to none. Thus $N_G(x)=N_G(y)$ for any two distinct
  support vertices $x,y$, contradicting \cref{eq:support-neighbourhoods}.

  \emph{Case B2 with $|\supp u_0|=2$.} Write
  $\supp u_0=\{x,y\}$ and $\supp u'_0=\{x',y'\}$.
  Since $u_0(x)=-u_0(y)\ne0$, we have
  $0=\langle u_0,u_i\rangle_G=u_0(x)(u_i(x)-u_i(y))$ for all $i\in I^*$,
  so $\tilde{r}_x=\tilde{r}_y$. The same holds for $u'_0$.
  By \cref{eq:support-row-sum},
  $\tilde{r}_x+\tilde{r}_y=\tilde{r}'_{x'}+\tilde{r}'_{y'}$, hence
  $\tilde{r}_x=\tilde{r}_y=\tilde{r}'_{x'}=\tilde{r}'_{y'}$.
  Together with \cref{eq:outside-rows}, this gives a vertex bijection
  $\pi$ satisfying $\tilde{r}'_{\pi(v)}=\tilde{r}_v$ for every $v\in V(G)$.
  Hence $G\cong G'$ by \cref{eq:row-isomorphism}.

  \emph{Case B2 with $|\supp u_0|=4$.}
  Both kernel vectors have support of size four and sum to zero.
  Their reduced rows outside the supports agree by \cref{eq:outside-rows}, and their
  reduced-row sums on the supports agree by \cref{eq:support-row-sum}.
  Thus \cref{lem:four-vertex-reconstruction}, proved below, identifies the remaining
  reduced rows as multisets.
  Combining the two correspondences gives a bijection $\pi$ with
  $\tilde{r}'_{\pi(x)}=\tilde{r}_x$, which is an isomorphism by \cref{eq:row-isomorphism}.
\end{proof}

\begin{lemma}\label{lem:four-vertex-reconstruction}
Let $G$ and $G'$ have nullity one and orthonormal eigenbases
$(u_i)_{i\in I}$ and $(u'_i)_{i\in I}$, respectively, with the same
corresponding eigenvalues $\lambda_i$.
Index both bases by $I=I^*\sqcup\{0\}$, with kernel vectors $u_0,u'_0$,
and write $\tilde{r}_x,\tilde{r}'_y\in\R^{I^*}$ for the reduced rows.
Suppose that $|\supp u_0|=|\supp u'_0|=4$ and
$\sum_xu_0(x)=\sum_yu'_0(y)=0$.
Suppose also that the reduced rows outside these supports agree as multisets and
$\sum_{x\in\supp u_0}\tilde{r}_x=\sum_{y\in\supp u'_0}\tilde{r}'_y$.
Then the reduced rows on the two supports agree as multisets.
\end{lemma}
\begin{proof}
  We first show that $\supp u_0$ admits at least two distinct
  partitions into pairs on which $u_0$ sums to zero.
  For $w\in V(G)$, let $N_0(w)=N_G(w)\cap\supp u_0$.
  Then $\sum_{x\in N_0(w)}u_0(x)=(A_Gu_0)(w)=0$.
  If a nonempty proper subset $T$ of $\supp u_0$ satisfies
  $\sum_{x\in T}u_0(x)=0$, then $|T|=2$ and
  $\sum_{x\in\supp u_0\setminus T}u_0(x)=0$, since the four support
  entries are nonzero and sum to zero.
  Choose distinct $x,y\in\supp u_0$.
  By \cref{eq:support-neighbourhoods}, there is a vertex $w_1$ for which
  $N_0(w_1)$ contains exactly one of $x,y$.
  Thus $N_0(w_1)$ is nonempty and proper, and
  $\{N_0(w_1),\supp u_0\setminus N_0(w_1)\}$ is a first partition
  into zero-sum pairs.
  The two vertices of $N_0(w_1)$ also have different neighbourhoods,
  so there is a vertex $w_2$ for which $N_0(w_2)$ contains exactly
  one of them. Again $N_0(w_2)$ is nonempty and proper, so
  $\{N_0(w_2),\supp u_0\setminus N_0(w_2)\}$ is a partition into
  zero-sum pairs. It differs from the first partition because its
  two pairs separate the vertices of $N_0(w_1)$.

  We label $\supp u_0=\{a,b,c,d\}$ so that $u_0$ sums to zero
  on each pair in the two partitions
  $\{\{a,c\},\{b,d\}\}$ and $\{\{a,d\},\{b,c\}\}$.
  Since $u_0$ sums to zero on each of $\{a,c\}$, $\{a,d\}$, and
  $\{b,c\}$, we have
  $u_0(c)=u_0(d)=-u_0(a)=-u_0(b)$.
  Since $u_0$ is a unit vector, a choice of sign gives
  \begin{equation}
    \label{eq:support4}
    \begin{aligned}
      u_0(a)&=u_0(b)=\frac12, &
      u_0(c)&=u_0(d)=-\frac12.
    \end{aligned}
  \end{equation}
  Since $\sum_yu'_0(y)=0$ and $|\supp u'_0|=4$, the same argument
  gives the nonzero entries $(\tfrac12,\tfrac12,-\tfrac12,-\tfrac12)$
  for $u'_0$, up to permutation and a choice of sign.

  We now show that $\{\tilde{r}_x:x\in\supp u_0\}$ and
  $\{\tilde{r}'_y:y\in\supp u'_0\}$ have the same centroid and span the
  same affine plane.
  The rows outside each support are orthonormal by \cref{eq:gram}.
  We relabel the vertices of $G'$ outside $\supp u'_0$ so that
  $\tilde{r}'_z=\tilde{r}_z$ for every $z\notin\supp u_0$.
  We choose the sign of $u'_0$ and label
  $\supp u'_0=\{a',b',c',d'\}$ so that
  $u'_0(a')=u'_0(b')=\tfrac12$ and
  $u'_0(c')=u'_0(d')=-\tfrac12$.
  Let $\sigma = \tilde{r}_a + \tilde{r}_b$,
  $f = (\tilde{r}_a - \tilde{r}_b)/\sqrt{2}$, and
  $g = (\tilde{r}_c - \tilde{r}_d)/\sqrt{2}$.
  Let $\sigma',f',g'$ be given by the formulas for $\sigma,f,g$ with
  primed vertices and rows.
  Since $u_0$ is orthogonal to every $u_i$ with $i\in I^*$,
  we have $\sum_xu_0(x)\tilde{r}_x=0$.
  Substituting \cref{eq:support4} gives
  $\tilde{r}_a+\tilde{r}_b=\tilde{r}_c+\tilde{r}_d$, so $\sigma=\tilde{r}_c+\tilde{r}_d$. Hence
  \begin{equation}
    \label{eq:special-rows}
    \begin{aligned}
    \tilde{r}_a &= \frac{1}{2} \sigma + \frac{1}{\sqrt{2}} f, &
    \tilde{r}_b &= \frac{1}{2} \sigma - \frac{1}{\sqrt{2}} f, &
    \tilde{r}_c &= \frac{1}{2} \sigma + \frac{1}{\sqrt{2}} g, &
    \tilde{r}_d &= \frac{1}{2} \sigma - \frac{1}{\sqrt{2}} g.
    \end{aligned}
  \end{equation}
  The centroids of $\{\tilde{r}_a,\tilde{r}_b,\tilde{r}_c,\tilde{r}_d\}$ and
  $\{\tilde{r}'_{a'},\tilde{r}'_{b'},\tilde{r}'_{c'},\tilde{r}'_{d'}\}$ are their averages
  $\sigma/2$ and $\sigma'/2$. By the assumed equality of the support-row sums,
  \begin{equation}\label{eq:quartet-offset}
    \frac12\sigma'=\frac14\sum_{y\in\supp u'_0}\tilde{r}'_y
    =\frac14\sum_{x\in\supp u_0}\tilde{r}_x=\frac12\sigma.
  \end{equation}
  By \cref{eq:gram,eq:support4}, the four reduced support rows have squared norm $3/4$,
  inner product $-1/4$ within each of the pairs $\{a,b\},\{c,d\}$, and
  inner product $1/4$ between the two pairs.
  Substituting these values into the definitions of $\sigma,f,g$ shows that
  these three vectors are orthonormal. They are orthogonal to the outside
  rows by \cref{eq:gram}. Together with
  $(\tilde{r}_z)_{z\in V(G)\setminus\supp u_0}$, they form an orthonormal
  family of $3+(|V(G)|-4)=|I^*|$ vectors, hence an orthonormal basis of $\R^{I^*}$.
  Since $\sigma'=\sigma$ and $\tilde{r}'_z=\tilde{r}_z$ for
  $z\notin\supp u_0$, the orthonormal bases
  $\{\sigma,f,g\}\cup\{\tilde{r}_z:z\notin\supp u_0\}$ and
  $\{\sigma',f',g'\}\cup\{\tilde{r}'_z:z\notin\supp u'_0\}$ give
  \begin{equation}\label{eq:quartet-plane}
   \spn(f,g)
   =\spn\bigl(\{\tilde{r}_z:z\notin\supp u_0\}\cup\{\sigma\}\bigr)^\perp
   =\spn(f',g').
  \end{equation}

  We now use the entries of $A_{G'}$ to show that $f'\in\{\pm f,\pm g\}$.
  Let $\alpha,\beta\in\R$ be the coefficients in $f'=\alpha f+\beta g$.
  Since $(f,g)$ is orthonormal and $f'$ is a unit vector,
  $\alpha^2+\beta^2=1$.
  For each $z\in V(G)\setminus\supp u_0$, let
  $\Delta_z^f=A_G(a,z)-A_G(b,z)$ and
  $\Delta_z^g=A_G(c,z)-A_G(d,z)$.
  By \cref{eq:rows,eq:special-rows},
  \begin{equation}\label{eq:quartet-adjacency}
    \begin{aligned}
      \sqrt2 f^{\top}\diag(\lambda_i)_{i\in I^*}\tilde{r}_z &= \Delta_z^f,\\
      \sqrt2 g^{\top}\diag(\lambda_i)_{i\in I^*}\tilde{r}_z &= \Delta_z^g,\\
      \sigma^{\top}\diag(\lambda_i)_{i\in I^*}\tilde{r}_z &= A_G(a,z)+A_G(b,z).
    \end{aligned}
  \end{equation}
  Since $\sigma'=\sigma$ and $f'=\alpha f+\beta g$, we have
  $\tilde{r}'_{a'}=\tfrac12\sigma+\tfrac1{\sqrt2}(\alpha f+\beta g)$
  by \cref{eq:special-rows} for $G'$.
  Substituting this formula for $\tilde{r}'_{a'}$ and $\tilde{r}'_z=\tilde{r}_z$ into
  the second identity in \cref{eq:rows} for $G'$, and using
  \cref{eq:quartet-adjacency}, gives
  \begin{equation}
    \label{eq:integrality}
    A_{G'}(a', z) = \frac{1}{2}\bigl(A_G(a,z)+A_G(b,z)\bigr)
      + \frac{1}{2}(\alpha\Delta_z^f+\beta\Delta_z^g)\in\{0,1\}.
  \end{equation}

  We have $\alpha\beta=0$ whenever a vertex $z\notin\supp u_0$ satisfies
  \begin{equation}\label{eq:quartet-probe}
    A_G(a,z)+A_G(b,z)=A_G(c,z)+A_G(d,z)=1.
  \end{equation}
  Indeed, then $\Delta_z^f,\Delta_z^g\in\{\pm1\}$ and
  $A_G(a,z)+A_G(b,z)=1$, so \cref{eq:integrality} gives
  $\alpha\Delta_z^f+\beta\Delta_z^g=\pm1$.
  Squaring gives
  $1=\alpha^2+\beta^2+2\alpha\beta\Delta_z^f\Delta_z^g
  =1+2\alpha\beta\Delta_z^f\Delta_z^g$, so $\alpha\beta=0$.
  \Cref{fig:quartet-geometry} gives a geometric interpretation of this calculation.

  \input{figures/quartet-geometry}

  We show that \cref{eq:quartet-probe} holds for some
  $z\in V(G)\setminus\supp u_0$.
  By \cref{eq:support4}, every $N_0(w)$ is empty, equal to $\supp u_0$, or one of
  $\{a,c\}$, $\{a,d\}$, $\{b,c\}$, $\{b,d\}$.
  Suppose no such $z$ exists. Then this list implies
  $N_0(z)\in\{\varnothing,\supp u_0\}$ for every
  $z\notin\supp u_0$.
  For $w\in\supp u_0$, we have $w\notin N_0(w)$ because $G$ has no loops.
  Thus a nonempty $N_0(w)$ must be one of
  $\{a,c\}$, $\{a,d\}$, $\{b,c\}$, and $\{b,d\}$.
  Every nonisolated vertex in the graph induced by $\supp u_0$
  therefore has exactly two neighbours, one in $\{a,b\}$ and one in
  $\{c,d\}$. A nonisolated vertex in $\{a,b\}$ is adjacent to the
  other vertex of $\{a,b\}$; the same holds for $\{c,d\}$.
  If the induced graph has an edge, it therefore also has an edge
  between $\{a,b\}$ and $\{c,d\}$.
  Each of these two sets then contains a nonisolated vertex and hence
  both its vertices are nonisolated. All four vertices have degree two,
  so the induced graph is a cycle of length four.
  Thus either the induced graph is edgeless, or it is a four-cycle.
  There are distinct $x,y\in\supp u_0$ with $N_0(x)=N_0(y)$:
  any two vertices in the edgeless case, or opposite vertices in the
  four-cycle case. Also
  $N_G(x)\setminus\supp u_0=N_G(y)\setminus\supp u_0$, because every
  vertex outside the support is adjacent to all four support vertices
  or to none. Hence $N_G(x)=N_G(y)$, contradicting
  $N_G(x)\ne N_G(y)$ for distinct $x,y\in\supp u_0$.

  A vertex $z\notin\supp u_0$ satisfying \cref{eq:quartet-probe}
  exists, so $\alpha\beta=0$. Together with $\alpha^2+\beta^2=1$,
  this gives $(\alpha,\beta)\in\{(1,0),(-1,0),(0,1),(0,-1)\}$,
  hence $f'\in\{\pm f,\pm g\}$. Since $(f',g')$ is orthonormal in
  $\spn(f,g)$, we have $\{\pm f',\pm g'\}=\{\pm f,\pm g\}$.
  Together with $\sigma'=\sigma$, \cref{eq:special-rows} now gives
  \begin{equation}\label{eq:quartet-row-equality}
    \{\tilde{r}_a,\tilde{r}_b,\tilde{r}_c,\tilde{r}_d\}
    =\{\tilde{r}'_{a'},\tilde{r}'_{b'},\tilde{r}'_{c'},\tilde{r}'_{d'}\}.
  \end{equation}
  The four vectors on each side are distinct by \cref{eq:special-rows},
  so they agree as multisets.
\end{proof}

\begin{proof}[Proof of \cref{thm:main}]

 Suppose $G\equiv_{\mathcal D_3}G'$. By \cref{lem:characterisation},
 $G\sim_{\mathrm{im}}G'$. Hence $G\cong G'$ by \cref{lem:reconstruction}.
 Conversely, an isomorphism $\pi:G\to G'$ gives a bijection
 $\varphi\mapsto\pi\circ\varphi$ between homomorphisms $F\to G$ and $F\to G'$
 for every $F$, so $G\equiv_{\mathcal D_3}G'$.
 A graph with simple spectrum has nullity at most one, because a real
 symmetric matrix with pairwise distinct eigenvalues has zero as an
 eigenvalue at most once. Hence $\mathcal D_3$ is complete on simple
 spectrum.
\end{proof}

%% file: figures/network-subdivision.tex
\begin{figure}[t]
\centering
\begin{tikzpicture}[
  x=.9cm,y=.9cm,
  tensor/.style={circle,fill=black,inner sep=2.5pt},
  adjacency/.style={rectangle,draw,fill=white,minimum size=6pt,inner sep=0pt},
  vertex/.style={circle,fill=black,inner sep=1.6pt},
  every node/.style={font=\small},
  line width=.6pt,
]
\node at (1.5,1.8) {Network on $\mathcal N$};
\draw (0,0)--(3,0);
\draw (0,0)..controls (.5,1.5) and (2.5,1.5)..(3,0)
  node[pos=.5,adjacency] {};
\draw (0,0)..controls (.5,-1.5) and (2.5,-1.5)..(3,0)
  node[pos=.3,adjacency] {} node[pos=.7,adjacency] {};
\node[tensor] at (0,0) {};
\node[tensor] at (3,0) {};
\node[adjacency] at (.15,-1.75) {};
\node[anchor=west] at (.35,-1.75) {$\langle\cdot,A_G\cdot\rangle_G$};
\node[tensor] at (.15,-2.3) {};
\node[anchor=west] at (.35,-2.3) {$\langle\cdot,\cdot,\cdot\rangle_G$};

\draw[->] (3.45,0)--(4.45,0);
\node[align=center,font=\scriptsize] at (3.95,-.58) {Suppress\\degree-two\\vertices};
\begin{scope}[xshift=4.5cm]
  \node at (1.5,1.8) {Cubic multigraph $\overline{\mathcal N}$};
  \draw (0,0)--(3,0) node[midway,above] {$q_e=0$};
  \draw (0,0)..controls (.5,1.5) and (2.5,1.5)..(3,0)
    node[midway,above] {$q_e=1$};
  \draw (0,0)..controls (.5,-1.5) and (2.5,-1.5)..(3,0)
    node[midway,below] {$q_e=2$};
  \node[vertex] at (0,0) {};
  \node[vertex] at (3,0) {};
  \node[align=center] at (1.5,-2.05) {$q_e$: adjacency forms\\on each path of $\mathcal N$};
\end{scope}

\draw[->] (8.45,0)--(9.45,0);
\node[align=center,font=\scriptsize] at (8.95,-.58) {Subdivide\\each edge\\($m_e=3$)};
\begin{scope}[xshift=9cm]
  \node at (1.5,1.8) {Simple graph $\overline{\mathcal N}^{(3,3,3)}$};
  \draw (0,0)--(3,0) node[pos=.333,vertex] {} node[pos=.667,vertex] {};
  \draw (0,0)..controls (.5,1.5) and (2.5,1.5)..(3,0)
    node[pos=.333,vertex] {} node[pos=.667,vertex] {};
  \draw (0,0)..controls (.5,-1.5) and (2.5,-1.5)..(3,0)
    node[pos=.333,vertex] {} node[pos=.667,vertex] {};
  \node[vertex] at (0,0) {};
  \node[vertex] at (3,0) {};
  \node[align=center] at (1.5,-2.05) {$m_e$ counts graph edges\\on the subdividing path};
\end{scope}
\end{tikzpicture}
\caption{A network on $\mathcal N$, its cubic multigraph, and one simple subdivision.
The path represented by $e$ has $q_e$ adjacency forms and $q_e+1$ edges in $\mathcal N$;
a direct edge has $q_e=0$.
The subdivision on the right is chosen independently of the network paths,
with $m_e=3$ on every edge.
The expansions of $\val_G(\mathcal N)$ and $\hom(\overline{\mathcal N}^{\mathbf m},G)$
use the same coefficients $c_G(\overline{\mathcal N},\ell)$, with weights
$\prod_e\ell(e)^{q_e}$ and $\prod_e\ell(e)^{m_e}$, respectively.
Varying the lengths $m_e\geq3$ determines these coefficients and hence
the network value for every choice of $q_e\geq0$.}
\label{fig:network-subdivision}
\end{figure}
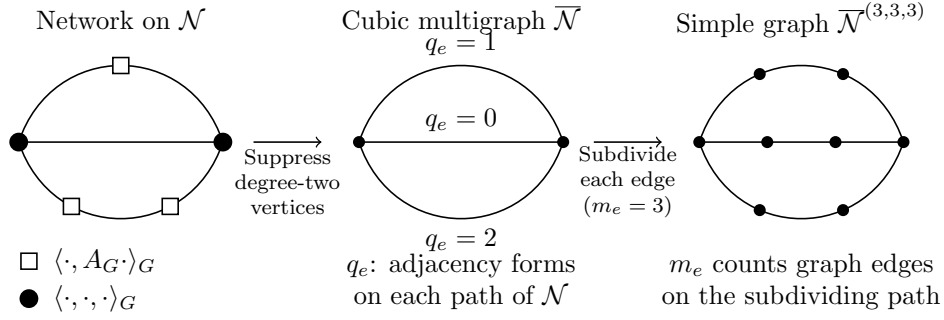

%% file: figures/quartet-geometry.tex
\begin{figure}[t]
\centering
\begin{tikzpicture}[
  every node/.style={font=\small},
  point/.style={circle,fill=black,inner sep=2pt},
  candidate/.style={circle,draw,fill=white,inner sep=2pt},
  line width=.6pt,
]
\node at (0,2.65) {Affine plane $\sigma/2+\spn(f,g)$};
\draw[densely dotted,->] (-1.95,0)--(2,0) node[right] {$f$};
\draw[densely dotted,->] (0,-1.95)--(0,2) node[above] {$g$};
\draw (0,0) circle[radius=1.35];
\draw[line width=.9pt] (-.6,1.95)--(1.95,-.6);
\node[anchor=west,font=\footnotesize] at (2,-.6) {adjacency to $z$: $1$};
\draw[line width=.9pt] (-1.95,.6)--(.6,-1.95);
\node[anchor=west,font=\footnotesize] at (.65,-1.95) {adjacency to $z$: $0$};
\node[fill=white,inner sep=1pt,anchor=north west] at (.05,-.05) {$\sigma/2$};
\node[point,label=above right:$\tilde{r}_a$] at (1.35,0) {};
\node[point,label=below left:$\tilde{r}_b$] at (-1.35,0) {};
\node[point,label=above right:$\tilde{r}_c$] at (0,1.35) {};
\node[point,label=below left:$\tilde{r}_d$] at (0,-1.35) {};
\node[candidate] (excluded) at (135:1.35) {};
\node[anchor=east,align=right,font=\footnotesize] at (-1.85,1.65)
  {excluded position\\for $\tilde{r}'_{a'}$};
\draw[thin] (-1.8,1.45)--(excluded);
\node at (0,-2.6) {Circle of radius $1/\sqrt2$};
\end{tikzpicture}
\caption{Geometric interpretation of the calculation giving $\alpha\beta=0$.
The figure assumes that $z\in V(G)\setminus\supp u_0$ is adjacent to $a,c$
and not to $b,d$; the other cases of \cref{eq:quartet-probe} are symmetric.
For a candidate $\tilde{r}'_{a'}=\sigma/2+(\alpha f+\beta g)/\sqrt2$,
the circle gives $\alpha^2+\beta^2=1$.
By \cref{eq:integrality}, the two lines have equations
$\alpha+\beta=1$ and $\alpha+\beta=-1$.
Squaring either equation and subtracting the circle equation gives
$2\alpha\beta=0$. The hollow point is on the circle but off both lines,
so \cref{eq:integrality} excludes it.}
\label{fig:quartet-geometry}
\end{figure}
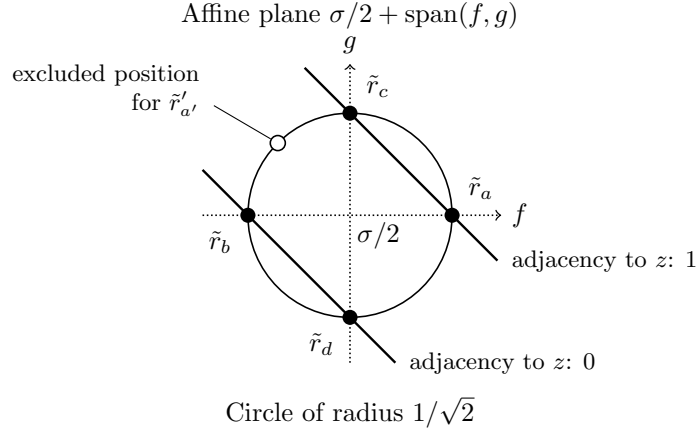

%% file: restriction.tex
\section{Proof of Theorem~\ref{thm:closure}}
\label{sec:restriction}

For a subdivision of a multigraph $H$ in a graph $X$, let
$\iota:V(H)\hookrightarrow V(X)$ denote its vertex map, which sends
each vertex of $H$ to its corresponding vertex in the subdivision.
For a non-loop edge $e=vw$, we write $R_e$ for the path replacing $e$,
with endpoints $\iota(v)$ and $\iota(w)$.
For a loop $e$ at $v$, we write $R_e$ for the cycle replacing $e$,
which passes through $\iota(v)$.

For a connected graph $K$, write $\rho_K$ for its spectral radius.
If $K$ has at least two vertices, then $\rho_K>0$.
For such $K$, the Perron--Frobenius theorem gives a unique positive unit
$\rho_K$-eigenvector $u_K$ and the identity
$P_{\rho_K,K}(x,y)=u_K(x)u_K(y)$.
Let $\tau_K=\sum_xu_K(x)^3>0$ and, for every nonzero real $\lambda$, let
\begin{equation}\label{eq:perron-data}
 M_{\lambda,K}=P_{\lambda,K}\diag(u_K)P_{\lambda,K}.
\end{equation}

It suffices to prove the theorem for a subsequence of $(H_N)$,
because the class obtained from a subsequence is contained in
$\calF\subseteq\mathcal D_3$.
By hypothesis, for every $t$ some $H_N$ contains $K_t$ as a minor.
Deleting finitely many terms preserves this property, since those terms
have bounded numbers of vertices.
We can therefore choose and reindex a subsequence so that $H_N$ contains
$K_N$ as a minor for every $N\geq1$.
Throughout this section, $(H_N)$ denotes this reindexed sequence.

For graphs $G$ and $G'$ with equal labelled coefficients
(\cref{eq:contraction}) on every $H_N$,
the main step is to recover equality on every connected cubic multigraph $H$
(\cref{lem:cubic-coefficient-transfer}).
We embed subdivisions of $H$ in the $H_N$ with independently chosen
path lengths (\cref{sec:subdivision-embedding}).
Using these subdivisions, we evaluate labelled coefficients on the $H_N$
and interpolate in both the subdivision path lengths and the numbers of
vertices of the $H_N$ (\cref{sec:coefficient-interpolation}).
We then compare the contributions from connected components of $G$ and
$G'$ in decreasing order of spectral radius (\cref{sec:component-removal}).

\subsection{Embedding subdivisions with independent path lengths}
\label{sec:subdivision-embedding}

In \cref{sec:coefficient-interpolation}, we recover labelled coefficients
on $H$ by interpolation.
To this end, we show that subdivisions of $H$ in $H_N$ can realise
various path lengths as follows.

\begin{lemma}\label{lem:independent-path-lengths}
 Let $H$ be a connected cubic multigraph and let $m\geq1$ be an integer.
 For every sufficiently large $N$, there are an injection
 $\iota:V(H)\hookrightarrow V(H_N)$ and finite sets
 $Q_e\subseteq\Z_{\geq1}$ with $|Q_e|\geq m$ for each $e\in E(H)$
 such that, for every $\mathbf q\in\prod_eQ_e$, there is a subdivision
 of $H$ in $H_N$ with vertex map $\iota$ in which each edge $e$ is replaced
 by a path or cycle $R_e$ of length $q_e+1$, and at least one edge of $H_N$
 lies outside the subdivision.
\end{lemma}

\begin{proof}
We construct an auxiliary cubic graph $D$ by replacing each edge of $H$
by a chain of $m$ diamonds.
The \emph{diamond} $K_4-ab$ is $K_4$ on vertices $a,b,c,d$ with the
edge $ab$ deleted. Its vertices $a,b$ have degree two, and $c,d$ have
degree three. A chain consists of $m$ disjoint copies together with
an edge from $b$ in each copy to $a$ in the next.
For each edge $vw$ of $H$, take a separate chain and add edges from
$v$ to $a$ in its first copy and from $b$ in its last copy to $w$.
The resulting graph $D$ is simple and cubic.
\Cref{fig:diamond-objects} illustrates the construction.
\input{figures/diamond-objects}

Every subdivision of a diamond has two $a$--$b$ paths of different lengths.
To see this, let $\len(acdb)$ denote the length of the path through
$a,c,d,b$, and use the same notation for the other paths. Then
\begin{equation}\label{eq:diamond-lengths}
 \len(acdb)+\len(adcb)
 =\len(acb)+\len(adb)+2\len(cd).
\end{equation}
Since $\len(cd)>0$, these four paths cannot all have the same length.

For $N\geq|V(D)|$, the simple graph $D$ is a subgraph of $K_N$,
so it is a minor of $H_N$.
Since $D$ is subcubic, it is a topological minor of
$H_N$~\cite[Section~1.7]{Diestel2017}.
Fix an isomorphism from a subdivision of $D$ onto a subgraph of $H_N$,
and let $\iota$ be its restriction to $V(H)$.
In each subdivided diamond in its image, choose two $a$--$b$ paths of
different lengths. For an edge $e=vw$ of $H$, we form $R_e$ by
concatenating one chosen path from each diamond in its chain with the
subdivided edges between consecutive diamonds and those from the chain's
ends to $\iota(v)$ and $\iota(w)$.
Switching from the shorter to the longer path in one copy at a time
gives at least $m$ distinct lengths for $R_e$.
Let $Q_e$ be the set of values $|E(R_e)|-1$ over these choices.
Every $R_e$ has an internal vertex, so $Q_e\subseteq\Z_{\geq1}$.
The paths and cycles for different edges have disjoint interiors, all avoiding
$\iota(V(H))$.
Since the diamond chains for different edges are disjoint, the path choices
in them are independent. Thus every $\mathbf q\in\prod_eQ_e$ can be realised by
a subdivision of $H$ in $H_N$ with the same vertex map $\iota$.
Every $a$--$b$ path in a diamond omits an edge of the diamond.
The corresponding subdivided edge in $H_N$ lies outside all chosen
paths and cycles $R_e$, proving the unused-edge assertion.
\end{proof}

\subsection{Recovering labelled coefficients by interpolation}
\label{sec:coefficient-interpolation}

To recover labelled coefficients on $H$ for $G$ and $G'$, we first analyse
the contributions from their connected components.
For a connected component $K$ with at least two vertices, the following lemma
evaluates its labelled coefficients on a cubic graph containing a subdivision
of $H$, with the edges outside the subdivision labelled by $\rho_K$.
The formula makes the dependence on subdivision path lengths and the number
of vertices of the containing graph explicit, providing the parameters for
interpolation.

\begin{lemma}\label{lem:perron-transfer}
 Let $K$ be a connected graph with at least two vertices, and use the notation of
 \cref{eq:perron-data}. Let $H$ be a cubic multigraph and let $X$ be a
 cubic graph containing a subdivision of $H$ with vertex map
 $\iota:V(H)\hookrightarrow V(X)$ in which each edge $e$ of $H$ is replaced
 by a path or cycle $R_e$.
 For each edge $e$ of $H$, let $q_e=|E(R_e)|-1$ and assume $q_e\geq1$.
 For a labelling $\ell:E(H)\to\R^\times$, let $\ell_X$ assign
 $\ell(e)$ to the edges of $R_e$ and $\rho_K$ to every other edge of $X$.
 Then
 \begin{equation}\label{eq:perron-transfer}
  c_K(X,\ell_X)=\tau_K^{|V(X)|-|V(H)|}
    \sum_{\varphi:V(H)\to V(K)}
    \prod_{e=vw\in E(H)}
      \left(\frac{M_{\ell(e),K}}{\tau_K}\right)^{q_e}(\varphi(v),\varphi(w)).
 \end{equation}
\end{lemma}

\Cref{fig:perron-objects} illustrates the graphs, paths, labels, and vertex maps
in \cref{lem:perron-transfer}.
\input{figures/perron-objects}

\begin{proof}
The vertices of $X$ consist of $\iota(V(H))$, the internal vertices of
the paths and cycles $R_e$, and the vertices outside all $R_e$.
Since $H$ and $X$ are cubic, the paths and cycles occupy all three edge incidences
at every vertex of $\iota(V(H))$.
An internal vertex of $R_e$ has exactly two incidences in $R_e$ and none
in any other $R_{e'}$, since the interiors are disjoint.
Its third incidence therefore lies outside the subdivision.
Every edge incident with a vertex outside all
$R_e$ lies outside the chosen subdivision.

In the sum defining $c_K(X,\ell_X)$, let $\psi:V(X)\to V(K)$ denote
the map being summed over and put $\varphi=\psi\circ\iota$.
For each edge $vw$ outside all $R_e$, factor its projector entry as
$P_{\rho_K,K}(\psi(v),\psi(w))=u_K(\psi(v))u_K(\psi(w))$.
Regrouping these factors by their endpoints gives no $u_K$ factor
at vertices of $\iota(V(H))$, one at each internal subdivision vertex, and three at each
vertex outside all $R_e$.
Summing over the image of each vertex outside all $R_e$ gives
$\sum_{x\in V(K)}u_K(x)^3=\tau_K$.
There are $|V(X)|-|V(H)|-\sum_eq_e$ such vertices.

On the paths and cycles $R_e$, retain the projector factors $P_{\ell(e),K}$,
including when $\ell(e)=\rho_K$.
For each edge $e=vw$, let $\lambda=\ell(e)$.
Summing over the image of one internal vertex of $R_e$ between fixed
neighbour images $x,y\in V(K)$ gives
\begin{equation}\label{eq:perron-path-factor}
 \sum_{z\in V(K)}P_{\lambda,K}(x,z)u_K(z)P_{\lambda,K}(z,y)
 =M_{\lambda,K}(x,y).
\end{equation}
Let $x_1,\ldots,x_{q_e}$ be the images of the internal vertices of $R_e$,
and put $x_0=\varphi(v)$ and $x_{q_e+1}=\varphi(w)$.
Since $P_{\lambda,K}^2=P_{\lambda,K}$ and $q_e\geq1$, summing over
these images gives
\begin{equation}\label{eq:path-image-sum}
 \begin{aligned}
 &\sum_{x_1,\ldots,x_{q_e}\in V(K)}
 \left(\prod_{j=0}^{q_e}P_{\lambda,K}(x_j,x_{j+1})\right)
 \left(\prod_{j=1}^{q_e}u_K(x_j)\right)\\
 &=\bigl[(P_{\lambda,K}\diag(u_K))^{q_e}P_{\lambda,K}\bigr]
   (\varphi(v),\varphi(w))\\
 &=M_{\lambda,K}^{q_e}(\varphi(v),\varphi(w)).
 \end{aligned}
\end{equation}
For a loop $e$, we have $\varphi(v)=\varphi(w)$, so the same calculation gives a diagonal
entry of $M_{\lambda,K}^{q_e}$.
Multiplying the contributions from the outside vertices and the $R_e$ for fixed $\varphi$,
then summing over $\varphi$, gives
\begin{equation}\label{eq:component-path-sum}
 c_K(X,\ell_X)
 =\tau_K^{|V(X)|-|V(H)|-\sum_eq_e}
   \sum_{\varphi:V(H)\to V(K)}
   \prod_{e=vw\in E(H)}M_{\ell(e),K}^{q_e}(\varphi(v),\varphi(w)).
\end{equation}
Substituting
$M_{\ell(e),K}^{q_e}=\tau_K^{q_e}(M_{\ell(e),K}/\tau_K)^{q_e}$
in \cref{eq:component-path-sum} proves \cref{eq:perron-transfer}.
The total exponent of $\tau_K$ is $|V(X)|-|V(H)|$, independent of the lengths of the $R_e$.
\end{proof}

The labelled coefficient $c_K(H,\ell)$ is a sum over vertex maps, with each
edge contributing an entry of the projector $P_{\ell(e),K}$.
The sum in \cref{eq:perron-transfer} uses entries of
$(M_{\ell(e),K}/\tau_K)^{q_e}$ for these factors.
We therefore seek to recover each projector as a linear combination of
these powers, with exponents determined by the available path lengths.

\begin{lemma}\label{lem:positive-interpolation}
Let $M_1,\ldots,M_p$ be positive semidefinite operators on finite-dimensional
Euclidean spaces, and let $P_i$ be the orthogonal projector onto $\im M_i$.
Suppose their spectra together contain at most $s$ distinct positive values,
and let $Q$ be a finite set of at least $s$ distinct positive integers.
Then there are real scalars $(a_q)_{q\in Q}$
with $\sum_{q\in Q}a_qM_i^q=P_i$ for every $i$.
\end{lemma}

\begin{proof}
Let $\mu_1,\ldots,\mu_{s'}$ be the distinct positive
values in the spectra of $M_1,\ldots,M_p$.
Since $s'\leq s$, we can choose distinct elements $q_1,\ldots,q_{s'}$ of $Q$.
The generalised Vandermonde matrix $(\mu_h^{q_j})_{h,j=1}^{s'}$ is
invertible~\cite[p.~76]{GantmacherKrein2002}.
Hence there are scalars $a_{q_1},\ldots,a_{q_{s'}}$ with
$\sum_ja_{q_j}\mu_h^{q_j}=1$ for every $h$; we set the remaining $a_q$ to zero.
The operator $\sum_{q\in Q}a_qM_i^q$ is the identity on each positive
eigenspace and is zero on $\ker M_i$, since every $q$ is positive.
It therefore equals $P_i$.
\end{proof}

Let $H$ be a connected multigraph and let $\ell$ be a nonzero
edge labelling.
Since $H$ is connected and spectral projector entries between distinct
connected components of $G$ are zero, a map $V(H)\to V(G)$ can contribute
to $c_G(H,\ell)$ only if its image lies in one connected component.
Grouping these contributions by that connected component gives the
additivity identity
\begin{equation}\label{eq:component-contributions}
 c_G(H,\ell)=\sum_K c_K(H,\ell),
\end{equation}
where $K$ ranges over the connected components of $G$.

The interpolation in subdivision path lengths leaves the factor
$\tau_K^{|V(X)|-|V(H)|}$ from \cref{eq:perron-transfer} in the contribution
from each connected component $K$ with at least two vertices.
For $X=H_N$, we remove these factors by a second interpolation in the number
of vertices of $H_N$.
The following lemma combines these steps to recover sums of labelled
coefficients over connected components.

\begin{lemma}\label{lem:two-stage-interpolation}
 Let $\mathcal K$ be a finite family of connected graphs with at least two
 vertices, and use the notation of \cref{eq:perron-data}.
 Let $H$ be a multigraph, let
 $\ell:E(H)\to\R^\times$ be an edge labelling, and let
 $\varepsilon_K\in\R$ for each $K\in\mathcal K$.
 Let $(d_N)$ be an unbounded sequence of nonnegative integers.
 Suppose that, for every sufficiently large $N$, there are finite sets
 $Q_{N,e}\subseteq\Z_{\geq1}$ with
 $|Q_{N,e}|\geq\sum_{K\in\mathcal K}|V(K)|$ for $e\in E(H)$ such that
 \begin{equation}\label{eq:two-stage-samples}
  \sum_{K\in\mathcal K}\varepsilon_K\tau_K^{d_N}
    \sum_{\varphi:V(H)\to V(K)}
    \prod_{e=vw\in E(H)}
      \left(\frac{M_{\ell(e),K}}{\tau_K}\right)^{q_e}(\varphi(v),\varphi(w))=0
 \end{equation}
 for every $\mathbf q\in\prod_eQ_{N,e}$.
 Then $\sum_{K\in\mathcal K}\varepsilon_Kc_K(H,\ell)=0$.
\end{lemma}

\begin{proof}
We first express $P_{\ell(e),K}$ as a linear combination
of powers of $M_{\ell(e),K}/\tau_K$, with the same scalars for all
$K\in\mathcal K$.
For every $K\in\mathcal K$ and $\lambda\ne0$, positivity of $u_K$ implies
that $M_{\lambda,K}$ is positive semidefinite. For $0\ne\xi\in E_{\lambda,K}$, we have
$\xi^\top M_{\lambda,K}\xi=\xi^\top\diag(u_K)\xi>0$.
The operator $M_{\lambda,K}$ is positive definite on $E_{\lambda,K}$ and zero
on $E_{\lambda,K}^{\perp}$, so its image is exactly $E_{\lambda,K}$.

For each edge $e$, the matrices $M_{\ell(e),K}/\tau_K$, for $K\in\mathcal K$,
are positive semidefinite and have image projectors $P_{\ell(e),K}$.
Their spectra together contain at most $\sum_{K\in\mathcal K}|V(K)|$
distinct positive values.
By \cref{lem:positive-interpolation}, for each sufficiently large $N$
and each edge $e$, there are scalars $a_{N,e,q}$ such that
\begin{equation}\label{eq:path-projector-interpolation}
 \sum_{q\in Q_{N,e}}a_{N,e,q}
   \left(\frac{M_{\ell(e),K}}{\tau_K}\right)^q
 =P_{\ell(e),K}
\end{equation}
for every $K\in\mathcal K$.
We multiply each instance of \cref{eq:two-stage-samples} by
$\prod_ea_{N,e,q_e}$ and sum over $\mathbf q\in\prod_eQ_{N,e}$.
The scalars $a_{N,e,q}$ are independent of $K$, so we may interchange
the sums and apply \cref{eq:path-projector-interpolation} on each edge.
By \cref{eq:contraction}, the resulting identity is
\begin{equation}\label{eq:two-stage-operators}
 \sum_{K\in\mathcal K}\varepsilon_K\tau_K^{d_N}c_K(H,\ell)=0.
\end{equation}

We next remove the factors $\tau_K^{d_N}$.
Since $d_N$ is unbounded, we can choose $|\mathcal K|$ indices $N_h$,
all sufficiently large, with distinct positive values of $d_{N_h}$.
Applying \cref{lem:positive-interpolation} to the operators
$t\mapsto\tau_Kt$ on $\R$, for $K\in\mathcal K$, with these exponents
gives scalars $b_h$ satisfying $\sum_hb_h\tau_K^{d_{N_h}}=1$ for every $K$.
Multiplying \cref{eq:two-stage-operators} at $N=N_h$ by $b_h$ and summing
over $h$ gives $\sum_{K\in\mathcal K}\varepsilon_Kc_K(H,\ell)=0$.
\end{proof}

\subsection{Comparing labelled coefficients}
\label{sec:component-removal}

We now prove \cref{lem:cubic-coefficient-transfer}, which extends equality
of labelled coefficients on the $H_N$ to every connected cubic multigraph.
We compare the combined contributions of connected components of maximum
spectral radius, remove those connected components, and repeat on the remaining
graphs.
\Cref{thm:closure} then follows from the characterisation in
\cref{lem:characterisation}(c).

\begin{lemma}\label{lem:maximum-radius-recovery}
 Let $G$ and $G'$ be finite simple graphs, at least one of which contains an edge.
 Suppose $c_G(H_N,\eta)=c_{G'}(H_N,\eta)$ for every $N\geq1$ and every
 labelling $\eta:E(H_N)\to\R^\times$.
 Let $r$ be the largest spectral radius among the connected components
 of $G$ and $G'$.
 Let $G^+$ and $G'^+$ be the disjoint unions of the connected components
 with spectral radius $r$ of $G$ and of $G'$, respectively.
 Then $c_{G^+}(H,\ell)=c_{G'^+}(H,\ell)$ for every connected cubic
 multigraph $H$ and every labelling $\ell:E(H)\to\R^\times$.
\end{lemma}

\begin{proof}
Fix a connected cubic multigraph $H$ and a labelling
$\ell:E(H)\to\R^\times$.
Let $\mathcal K$ be the family of connected components of $G^+$ and $G'^+$,
counted separately, and put $\varepsilon_K=1$ for those from $G^+$ and
$\varepsilon_K=-1$ for those from $G'^+$.
By \cref{eq:component-contributions}, our goal is
$\sum_{K\in\mathcal K}\varepsilon_Kc_K(H,\ell)=0$.

Put $m=|V(G)|+|V(G')|$ and $d_N=|V(H_N)|-|V(H)|$.
Since $H_N$ contains $K_N$ as a minor, $|V(H_N)|\geq N$, so $d_N$
is unbounded and eventually nonnegative.
Choose $N_0$ so that $d_N\geq0$ and \cref{lem:independent-path-lengths}
applies with this $m$ whenever $N\geq N_0$.
For each such $N$, write $\iota_N$ and $Q_{N,e}$ for the injection and sets
from that lemma, recording their dependence on $N$.
Thus $|Q_{N,e}|\geq m$.
For each $\mathbf q\in\prod_eQ_{N,e}$, choose a subdivision with this
vertex map and with each $R_e$ having length $q_e+1$.
Let $\ell_N(\mathbf q)$ assign $\ell(e)$ to the edges of $R_e$ and
$r$ to every edge outside the subdivision.

At least one edge lies outside the subdivision, by
\cref{lem:independent-path-lengths}, and receives the label $r$.
Every connected component $K$ of $G$ or $G'$ with $\rho_K<r$ has $P_{r,K}=0$,
so only the connected components of $G^+$ contribute to $c_G(H_N,\ell_N(\mathbf q))$,
and only those of $G'^+$ contribute to $c_{G'}(H_N,\ell_N(\mathbf q))$.
For every $K\in\mathcal K$, the labelling $\ell_N(\mathbf q)$ is the
labelling in \cref{lem:perron-transfer}, since $\rho_K=r$.
We apply that lemma to these connected components and sum as in
\cref{eq:component-contributions}.
Subtracting the evaluations of $c_G(H_N,\ell_N(\mathbf q))$ and
$c_{G'}(H_N,\ell_N(\mathbf q))$, which agree by hypothesis, gives
\begin{equation}\label{eq:operator-transfer}
 \sum_{K\in\mathcal K}\varepsilon_K\tau_K^{d_N}
   \sum_{\varphi:V(H)\to V(K)}
   \prod_{e=vw\in E(H)}
     \left(\frac{M_{\ell(e),K}}{\tau_K}\right)^{q_e}(\varphi(v),\varphi(w))=0.
\end{equation}
We apply \cref{lem:two-stage-interpolation} to $\mathcal K$, using $N\geq N_0$.
Every $K\in\mathcal K$ has spectral radius $r>0$ and hence at least two vertices.
The graphs in $\mathcal K$ have at most $m$ vertices in total, so
$|Q_{N,e}|\geq m\geq\sum_{K\in\mathcal K}|V(K)|$ for every $e$.
The sequence $(d_N)_{N\geq N_0}$ is nonnegative and unbounded, and
\cref{eq:operator-transfer} is the hypothesis \eqref{eq:two-stage-samples}.
The lemma therefore gives $\sum_{K\in\mathcal K}\varepsilon_Kc_K(H,\ell)=0$,
as required.
\end{proof}

\begin{lemma}\label{lem:cubic-coefficient-transfer}
 Let $G$ and $G'$ be finite simple graphs. Suppose
 $c_G(H_N,\eta)=c_{G'}(H_N,\eta)$ for every $N\geq1$ and every
 labelling $\eta:E(H_N)\to\R^\times$.
 Then $c_G(H,\ell)=c_{G'}(H,\ell)$ for every connected cubic
 multigraph $H$ and every labelling $\ell:E(H)\to\R^\times$.
\end{lemma}

\begin{proof}
We use induction on the total number of connected components in $G$ and $G'$.
If both graphs are edgeless, all their projectors at nonzero eigenvalues
are zero. Since every connected cubic multigraph has an edge, all its
labelled coefficients are then zero.
Otherwise, let $r>0$ be the largest spectral radius among their connected components.
Let $G^+$ and $G'^+$ be the disjoint unions of the connected components
with spectral radius $r$ of $G$ and of $G'$, respectively.
By \cref{lem:maximum-radius-recovery},
$c_{G^+}(H,\ell)=c_{G'^+}(H,\ell)$ for every connected cubic multigraph $H$
and every nonzero edge labelling $\ell$.

Let $G^-$ and $G'^-$ be obtained from $G$ and $G'$, respectively,
by deleting these connected components.
For every $N$ and every labelling $\eta:E(H_N)\to\R^\times$, the equality
just proved applies to $H_N$, since $H_N$ is connected and cubic.
Together with the hypothesis and \cref{eq:component-contributions}, it gives
\begin{equation}\label{eq:residual-contractions}
 \begin{aligned}
  c_{G^-}(H_N,\eta)
   &=c_G(H_N,\eta)-c_{G^+}(H_N,\eta)\\
   &=c_{G'}(H_N,\eta)-c_{G'^+}(H_N,\eta)
    =c_{G'^-}(H_N,\eta).
 \end{aligned}
\end{equation}
At least one connected component is deleted, so $G^-$ and $G'^-$ have fewer
connected components in total.
The induction hypothesis therefore gives
$c_{G^-}(H,\ell)=c_{G'^-}(H,\ell)$ for every connected cubic multigraph $H$
and every nonzero edge labelling $\ell$.
Adding the equal labelled coefficients for $G^+,G'^+$ and for $G^-,G'^-$,
and using \cref{eq:component-contributions}, gives
$c_G(H,\ell)=c_{G'}(H,\ell)$.
\end{proof}

\begin{proof}[Proof of \cref{thm:closure}]
Let $\calF'$ consist of $K_1$ and all subdivisions $H_N^{\mathbf m}$ of
the reindexed sequence, using the notation of \cref{sec:subcubic-characterisation}.
Let $G$ and $G'$ satisfy $G\equiv_{\calF'}G'$. Since $K_1\in\calF'$, we have
$|V(G)|=|V(G')|$.
We also have $\hom(H_N^{\mathbf m},G)=\hom(H_N^{\mathbf m},G')$
for every $N$ and every simple subdivision $H_N^{\mathbf m}$.
By \cref{lem:subdivision-coefficients},
$c_G(H_N,\ell)=c_{G'}(H_N,\ell)$ for every $N$ and every nonzero labelling $\ell$.
\Cref{lem:cubic-coefficient-transfer} gives
$c_G(H,\ell)=c_{G'}(H,\ell)$ for every connected cubic multigraph $H$
and every nonzero edge labelling $\ell$.
By \cref{lem:characterisation}(c),
$G\equiv_{\mathcal D_3}G'$.

The reverse implication follows from $\calF'\subseteq\mathcal D_3$.
Since $\calF'\subseteq\calF\subseteq\mathcal D_3$, this also proves
that the relations $\equiv_{\calF}$ and $\equiv_{\mathcal D_3}$ coincide.
Completeness on simple spectrum follows from \cref{thm:main}.
\end{proof}

%% file: figures/diamond-objects.tex
\begin{figure}[t]
\centering
\begin{tikzpicture}[
  x=.94cm,y=.94cm,
  every node/.style={font=\small},
  vertex/.style={circle,fill=black,inner sep=1.35pt},
  branch/.style={circle,fill=black,inner sep=2pt},
  edge/.style={draw=black!50,line width=.55pt},
  chosen/.style={draw=black,line width=1.3pt},
]
\node at (0,2.15) {$H$};
\node at (4.3,2.15) {$D$ ($m=1$)};
\node at (9.4,2.15) {$X$};
\node[branch,label=left:$v$] (hu) at (-.7,0) {};
\node[branch,label=right:$w$] (hv) at (.7,0) {};
\draw (hu) .. controls (-.7,.8) and (.7,.8) .. (hv)
  node[midway,above] {$e$};
\draw (hu)--(hv) node[midway,above] {$f$};
\draw (hu) .. controls (-.7,-.8) and (.7,-.8) .. (hv)
  node[midway,below] {$g$};

\foreach \offset/\name in {4.3/d,9.4/x} {
  \begin{scope}[xshift=\offset*.94cm]
    \coordinate (\name u) at (-1.7,0);
    \coordinate (\name v) at (1.7,0);
    \foreach \height/\j in {1.2/e,0/f,-1.2/g} {
      \coordinate (\name a\j) at (-.8,\height);
      \coordinate (\name b\j) at (.8,\height);
      \coordinate (\name c\j) at (0,\height+.35);
      \coordinate (\name d\j) at (0,\height-.35);
      \draw[edge] (\name u)--(\name a\j)--(\name c\j)--(\name b\j)--(\name v);
      \draw[edge] (\name a\j)--(\name d\j)--(\name b\j);
      \draw[edge] (\name c\j)--(\name d\j);
    }
  \end{scope}
}
\node[branch,label=left:$v$] at (du) {};
\node[branch,label=right:$w$] at (dv) {};
\node[branch,label=left:$\iota(v)$] at (xu) {};
\node[branch,label=right:$\iota(w)$] at (xv) {};
\coordinate (xs) at (9.4,1.2);
\draw[chosen] (xu)--(xae)--(xce)--(xs)--(xde)--(xbe)--(xv);
\draw[chosen] (xu)--(xaf)--(xcf)--(xbf)--(xv);
\draw[chosen] (xu)--(xag)--(xcg)--(xbg)--(xv);
\foreach \name in {d,x} {
  \foreach \j in {e,f,g} {
    \foreach \p in {a,b,c,d} {\node[vertex] at (\name\p\j) {};}
  }
}
\node[vertex] at (xs) {};
\foreach \height/\j in {1.2/e,0/f,-1.2/g} {
  \node at (9,\height+.48) {$R_\j$};
}
\node[align=left,font=\footnotesize,anchor=west] at (10.5,1.5)
  {$Q_e=\{3,5\}$\\$q_e=5$};
\draw[->] (1.25,1.65)--(2.5,1.65)
  node[midway,above,font=\footnotesize] {replace edges};
\draw[->] (6.35,1.65)--(7.55,1.65)
  node[midway,above,font=\footnotesize] {subdivide};
\draw[->] (0,-2)--(9.4,-2)
  node[midway,fill=white,inner sep=3pt] {$\iota:V(H)\hookrightarrow V(X)$};
\end{tikzpicture}
\caption{The construction in \cref{lem:independent-path-lengths}.
Each edge of $H$ is replaced by a diamond in $D$.
The graph $X$ represents a subdivided copy of $D$ inside $H_N$;
the vertical edge of its top diamond is subdivided once.
The bold paths form the chosen subdivision of $H$.
The paths for $e$ have four or six edges, so $Q_e$ lists their possible
internal-vertex counts.}
\label{fig:diamond-objects}
\end{figure}
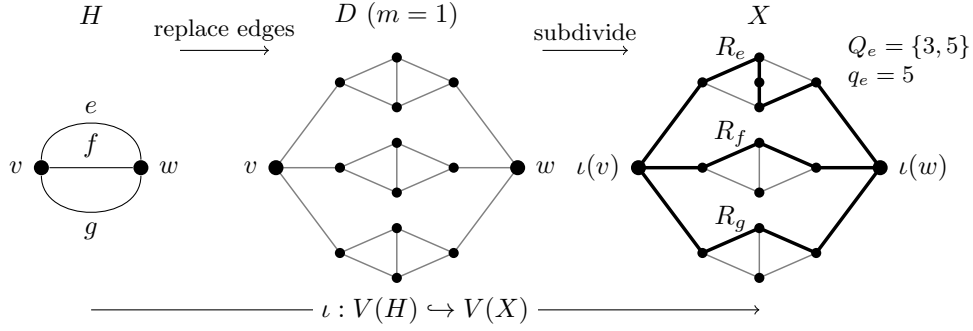

%% file: figures/perron-objects.tex
\begin{figure}[t]
\centering
\begin{tikzpicture}[
  x=.96cm,y=.96cm,
  every node/.style={font=\small},
  vertex/.style={circle,fill=black,inner sep=1.55pt},
  branch/.style={circle,fill=black,inner sep=2.2pt},
  outside/.style={circle,draw=black,fill=white,inner sep=2.2pt},
  route/.style={line width=1.1pt},
  perron/.style={densely dashed,line width=.65pt},
]
\node at (0,2.15) {$H$};
\node at (5.3,2.15) {$X$};
\node at (10.4,2.15) {$K$};
\node[branch,label=left:$v$] (hu) at (-.7,0) {};
\node[branch,label=right:$w$] (hv) at (.7,0) {};
\draw (hu) .. controls (-.7,.8) and (.7,.8) .. (hv)
  node[midway,above] {$e:\ell(e)$};
\draw (hu)--(hv) node[midway,above] {$f:\ell(f)$};
\draw (hu) .. controls (-.7,-.8) and (.7,-.8) .. (hv)
  node[midway,below] {$g:\ell(g)$};

\coordinate (xu) at (3.6,0);
\coordinate (xv) at (7,0);
\foreach \height/\j in {1.2/e,0/f,-1.2/g} {
  \coordinate (a\j) at (4.5,\height);
  \coordinate (b\j) at (6.1,\height);
  \coordinate (c\j) at (5.3,\height+.35);
  \coordinate (d\j) at (5.3,\height-.35);
  \draw[perron] (a\j)--(d\j)--(b\j);
  \draw[perron] (c\j)--(d\j);
  \draw[route] (xu)--(a\j)--(c\j)--(b\j)--(xv);
  \foreach \p in {a,b,c} {\node[vertex] at (\p\j) {};}
  \node[outside] at (d\j) {};
  \node at (4.9,\height+.48) {$R_\j$};
}
\node[branch,label=left:$\iota(v)$] at (xu) {};
\node[branch,label=right:$\iota(w)$] at (xv) {};

\draw[black!50] (10.4,0) ellipse (1.4 and .9);
\node[vertex,label=below:$\varphi(v)$] at (9.8,.15) {};
\node[vertex,label=above:$\varphi(w)$] at (11,-.15) {};
\draw[->] (1.25,.55)--(3.15,.55) node[midway,above] {$\iota$};
\draw[->] (7.75,.55)--(9.25,.55) node[midway,above] {$\psi$};
\draw[->] (0,-2)--(10.4,-2)
  node[midway,fill=white,inner sep=3pt] {$\varphi$};
\end{tikzpicture}
\caption{The objects in \cref{lem:perron-transfer}.
Edges of $R_e$ have label $\ell(e)$; each solid path has three internal vertices.
Dashed edges have label $\rho_K$; hollow vertices lie outside all $R_e$.
The target $K$ is schematic, with two vertex images marked.
The coefficient $c_K(X,\ell_X)$ sums over all maps $\psi:V(X)\to V(K)$;
the right side of \cref{eq:perron-transfer} sums over all
$\varphi:V(H)\to V(K)$.
For a fixed $\psi$, we have $\varphi=\psi\circ\iota$.}
\label{fig:perron-objects}
\end{figure}
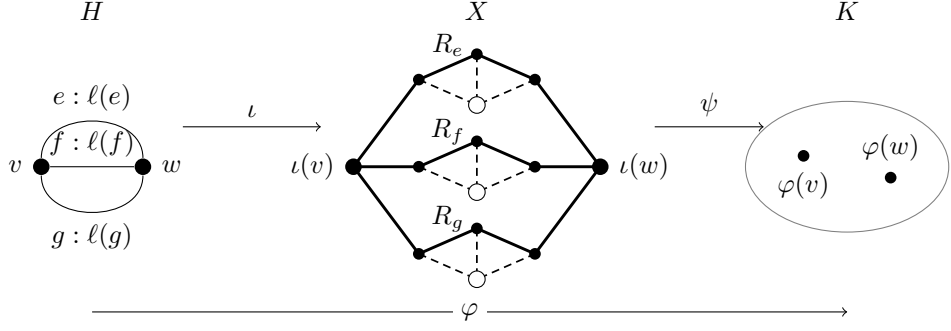

%% file: open-problems.tex
\section{Open Problems}
\label{sec:open-problems}

\Cref{thm:incompleteness} shows that every class complete on simple
spectrum contains a graph with a $K_7$ minor.
For the subdivision families in \cref{thm:closure}, arbitrarily large
clique minors suffice for completeness.

\input{figures/cricket}

\begin{problem}\label{prob:clique-minors}
Must every graph class complete on simple spectrum contain, for each
$t\geq1$, a graph with a $K_t$ minor?
\end{problem}

An affirmative answer to \cref{prob:clique-minors} would imply the weak
Roberson conjecture, which states that every proper minor-closed graph
class $\mathcal C$ admits non-isomorphic graphs
$G\equiv_{\mathcal C}G'$~\cite[Conjecture~5 and Theorem~8.4]{Roberson2026}.
Every proper minor-closed class $\mathcal C$ excludes some $K_t$ and would
therefore fail to be complete on simple spectrum, yielding such a pair
with simple spectrum.
To prove the necessity in \cref{prob:clique-minors}, it suffices by
\cref{lem:ss-transfer} to find, for each $t\geq1$, a connected graph $B$
whose CFI pair is homomorphism indistinguishable
over the class of $K_t$-minor-free graphs.
The proof of \cref{thm:incompleteness} uses such a pair for $t=7$
from~\cite{NeuenSeppelt2026,ThilikosWiederrecht2024}.

\Cref{thm:main} leaves the case of nullity two open.
\begin{problem}\label{prob:nullity-two}
If $G$ has nullity two, does $G\equiv_{\mathcal D_3}G'$ imply
$G\cong G'$ for every graph $G'$?
\end{problem}

The nullity bound cannot be raised to three.
Let $B$ be the cricket graph, obtained from a triangle by attaching two
leaves to the same vertex (\cref{fig:cricket}).
The graphs $\CFI_0(B)$ and $\CFI_1(B)$ defined in
\cref{eq:ss-cfi-definition} each have $14$ vertices.
A direct calculation gives nullity three for both graphs.
They are non-isomorphic and homomorphism indistinguishable over
$\mathcal D_3$ by~\cite[Corollary~3.7 and the proof of Corollary~4.2]{Roberson2026}.

On pairs of graphs with simple spectrum, subcubic homomorphism
indistinguishability can be tested in polynomial time by \cref{thm:main}
and~\cite{BabaiGrigoryevMount1982}.
On arbitrary graphs, this relation is decidable
by~\cite[Corollary~39]{CaiYoung2026}.
However, deciding this relation is at least as hard as graph isomorphism
under polynomial-time many-one reductions.

\begin{problem}\label{prob:polynomial-testing}
Is there a class $\mathcal C\subseteq\mathcal D_3$, complete on simple
spectrum, such that deciding $G\equiv_{\mathcal C}G'$ for arbitrary
graphs $G$ and $G'$ is possible in polynomial time?
\end{problem}

To prove the hardness of deciding $G\equiv_{\mathcal D_3}G'$, we reduce
from graph isomorphism on connected graphs with at least two
vertices; this restriction is as hard as general graph
isomorphism~\cite[Section~16]{ZemlyachenkoKorneenkoTyshkevich1985}.
For such a graph $G$, let $G^\star$ be obtained by attaching a leaf to every
vertex. Then $G^\star$ is nonsingular: a vector in its adjacency kernel
vanishes on the original vertices by the equations at the leaves, and then
on the leaves by the equations at the original vertices.
The vertices of degree at least two in $G^\star$ induce $G$.
Thus two such graphs $G,G'$ satisfy $G\cong G'$ if and only if
$G^\star\equiv_{\mathcal D_3}G'^\star$, by \cref{thm:main}
or~\cite[Corollary~38]{CaiYoung2026}.
The construction of $G^\star$ takes polynomial time.

The families in \cref{thm:closure} define the same indistinguishability
relation as $\mathcal D_3$ on all graphs, so they inherit its
graph-isomorphism hardness. \Cref{prob:polynomial-testing} allows
$\mathcal C$ to distinguish fewer graph pairs than $\mathcal D_3$,
provided it remains complete on simple spectrum.

%% file: figures/cricket.tex
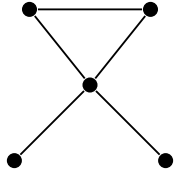
\begin{figure}[t]
\centering
\begin{tikzpicture}[
  vertex/.style={circle,fill,inner sep=2pt},
  line width=.7pt,
]
\node[vertex] (centre) at (0,0) {};
\node[vertex] (upperleft) at (-.8,1) {};
\node[vertex] (upperright) at (.8,1) {};
\node[vertex] (lowerleft) at (-1,-1) {};
\node[vertex] (lowerright) at (1,-1) {};
\draw (centre)--(upperleft)--(upperright)--(centre);
\draw (lowerleft)--(centre)--(lowerright);
\end{tikzpicture}
\caption{The cricket graph $B$.}
\label{fig:cricket}
\end{figure}